\documentclass[reqno]{amsart}
\usepackage[T1]{fontenc}
\usepackage[utf8]{inputenc}
\usepackage[hang]{caption}

\usepackage{enumerate,amsmath,amssymb, geometry}
\usepackage{graphicx, soul,footnote}
\usepackage{xcolor,fullpage, tikz}
\usepackage{hyperref}
\theoremstyle{plain}      
 
\newtheorem{thm}{Theorem}[section]     
\newtheorem{theorem}[thm]{Theorem}     
\newtheorem{cor}[thm]{Corollary}     
\newtheorem{corollary}[thm]{Corollary}     
     
\newtheorem{lemma}[thm]{Lemma}
\newtheorem{prop}[thm]{Proposition}     
\newtheorem{proposition}[thm]{Proposition}

\theoremstyle{definition}      
\newtheorem{defi}[thm]{Definition}

\newtheorem{remark}[thm]{Remark}
\newtheorem{example}[thm]{Example}

\renewcommand{\epsilon}{\varepsilon}
\let\theta\vartheta
\let\phi\varphi
\let\th\theta
\let\ga\gamma

\let\<\langle
\let\>\rangle
\let\doba\mathbb

\def\mH{\doba{H}}
\def\mN{\doba{N}}
\def\mM{\doba{M}}

\def\mR{\doba{R}}
\def\mS{\doba{S}}

 \let\R\mR

\def\cF{\mathcal{F}}
\def\cFs{\cF^{\rm spin}}

\def\scal{{\mathop{\rm scal}}}

\def\Spin{{\mathop{\rm Spin}}}

\def\supp{{\mathop{\rm supp}}}
\def\vol{{\mathop{\rm vol}}}

\def\dom{{\mathop{\rm dom}}}

\let\pa\partial

\let\la\lambda
\let\na\nabla
\let\ep\epsilon
\newcommand{\lm}{\lambda^{+,*}_{\mathrm{min}}}
\newcommand{\wlm}{\widetilde{\lambda}^+_{\mathrm{min}}}
\newcommand{\Qs}{{Q_{\rm spin}^{*}}}
\newcommand{\Q}{{Q^{*}}}
\newcommand{\wQs}{{\widetilde{Q}_{\rm spin}}}
\newcommand{\wQ}{{\widetilde{Q}}}
\newcommand{\sLa}{\Lambda^*}
\newcommand{\La}{\Lambda}
\newcommand{\wLa}{\widetilde{\Lambda}}
\newcommand{\sLas}{\Lambda^{\rm spin, *}}
\newcommand{\Las}{\Lambda^{\rm spin}}
\newcommand{\wLas}{\widetilde{\Lambda}^{\rm spin}}
\newcommand{\sis}{{\sigma_{\rm spin}^*}}  
\newcommand{\si}{{\sigma^*}}

\newcommand{\dist}{\mathrm{dist}}
\newcommand{\vo}{\mathrm{dvol}}
\renewcommand{\d}{\mathrm{d}}
\renewcommand{\Mc}{\mathbb{M}_c}

\renewcommand{\Re}{\mathrm{Re}}

\begin{document}

\title{Relations between threshold constants for {Y}amabe type bordism invariants}

\author{Bernd Ammann} 
\address{Fakult\"at f\"ur Mathematik, Universit\"at Regensburg, 93040
Regensburg, Germany}
\email{bernd.ammann@mathematik.uni-regensburg.de}

\author{Nadine Gro\ss e} 
\address{Institut f\"ur Mathematik, Universit\"at Leipzig, 04109 Leipzig,
Germany}
\email{grosse@math.uni-leipzig.de}



\begin{abstract}
In the work of Ammann, Dahl and Humbert it has turned out that the Yamabe invariant on closed manifolds is a bordism invariant below a certain threshold constant. A similar result holds for  a spinorial analogon. These threshold constants are characterized through Yamabe-type equations on products of spheres with rescaled hyperbolic spaces. We give variational characterizations of these threshold constants, and our investigations lead to an explicit positive lower bound for the spinorial threshold constants.
\end{abstract}

\maketitle

\section{Introduction}

The smooth Yamabe invariant, also called Schoen's $\sigma$-constant, of a closed manifold $M^m$ is defined as
\[\si (M) := \sup \inf \int_M \scal^g \vo_g\in \Big(-\infty, \si(\mS^m) =m(m-1) \vol(\mS^m)^{\frac{2}{m}}\Big],\]
where the supremum runs over all conformal classes $[g_0]$ on $M$, and the
infimum goes over all metrics $g$ in $[g_0]$ such that $(M,g)$ has volume one. The smooth Yamabe invariant is an important geometric quantity. In particular, $\si(M)>0$ if and only if $M$ admits a metric of positive scalar curvature. However,the smooth Yamabe invariant is quite mysterious. It is only known for very few examples, e.g. the sphere, cp. Remark~\ref{inv_sphere},  $\si (\mathbb{T}^m)=0$, cf. \cite[Cor.~2.5]{grom_law}, and $\si(\R\mathbb{P}^3)=2^{-2/3}\si(\mS^3)$, \cite[Cor.~2.3]{bray_neves}. In particular, there is no known example of a manifold of dimension $m\geq 5$ with $\si(M)\not\in \{0, \si(\mS^m)\}$. 

In \cite{ADH} M.~Dahl, E.~Humbert, and the first author proved a surgery formula for the smooth Yamabe invariant, cp. Theorem~\ref{surgthm}. In particular it says that if a manifold $N^m$ is obtained from a closed manifold $M^m$ by a surgery of codimension $m-k\geq 3$ and $\si(M)$ is below a certain threshold constant $\La_{m,k}$, then $\si(N)\geq \si(M)$.  The threshold constants $\Lambda_{m,k}$ appearing in this result are certain Yamabe-type invariants for special noncompact model spaces $\Mc^{m,k}$ which are products of rescaled hyperbolic spaces and spheres, see Subsection~\ref{mod_sp} for the precise definition. 
In particular, it follows that the smooth Yamabe invariant is a bordism invariant in the following sense: 
Suppose that $M$ and $N$ are connected closed smooth spin manifolds of dimension $m\geq 5$ with fundamental group $\Gamma$, representing the same element in $\Omega^{\rm spin}_m (B\Gamma)$, then $0\leq \si(M) < \La_m:= \min_{k=2,\ldots, m-3} \La_{m,k}$ implies $\si (M)=\si(N)$, \cite[Sec.~1.4]{ADH}. Thus, if sufficiently many manifolds with $\si(M)\in (0, \La_m)$ exist, one obtains a rich and interesting subgroup in the bordism group $\Omega^{\rm spin}_m (B\Gamma)$ and similar versions hold in oriented bordism classes.

In order to understand the structure of  subgroup it is essentially to get as much knowledge about the surgery constants $\La_{m,k}$ as possible.

If current conjectures about explicit lower bounds for $\La_{m,k}$, see \cite[Sec.~1.4]{ADH}, turn out to be true, then the supremum in the smooth Yamabe invariant of $\mathbb{CP}^3$ is not attained by the Fubini-Study metric.

The current article will not give explicit positive lower bound for $\La_{m,k}$, but it will provide many relations to a spinorial analogue of the problem. The smooth Yamabe invariant $\si(M)$ has a spinorial analogue $\sis(M)$, cf. Subsection~\ref{mod_sp}. For closed  manifolds the Hijazi inequality  gives $\sis(M)\geq \si(M)$.
 As in the Yamabe case there is a surgery formula for $\sis$, cp. Theorem~\ref{surgspinthm}, and again a threshold constant $\Las_{m,k}$ appears. However, now even codimension $2$ surgeries are covered. This has implications for the smooth Yamabe constant as well: If  $M$ and $N$ are arbitrary closed spin manifolds (not necessarily simply connected), and if $M$ is spin-bordant to $N$, then $\sis(M)<\Las_{m,k}$ for $k=0,\ldots, m-2$ implies $\sis(M)=\sis(N)$. In particular, $\sis(M)\geq \si(N)$. Finding interesting manifolds with $\sis(M)<\Las_{m,k}$ consists of two parts. First one has to obtain explicit positive lower bounds for $\Las_{m,k}$ which is the main subject of the present article and then finding examples for $M$ which is not covered here.
 
As the threshold constant are defined as (spinorial) Yamabe-type invariants of noncompact model spaces, one expects in view of the Hijazi inequality that  $\Las_{m,k}\geq \La_{m,k}$.  This question is quite subtle because on noncompact manifolds there are several ways to define Yamabe-type invariants which are sometimes related and sometimes unrelated to each other. One goal of the article is to clarify these relations. 

The structure of the article: In Section~\ref{sec_prelim} we fix notations, preliminaries and give existing results. In particular, we define the model spaces and the different (spinorial) Yamabe-type invariants for noncompact manifolds. 
This allows us to summarize the results of the article in Section~\ref{sec_results}. These results are proved in the remaining sections. In particular, in Section~\ref{sec_reg} we provide regularity statement for the Euler-Lagrange equation of the spinorial Yamabe functional, which is a nonlinear Dirac eigenvalue equation. For more details we refer to the end of Section~\ref{sec_reg}.\\

{\bf Acknowledgment.} The second author thanks the University of Regensburg for the hospitality during a short term visit in Regensburg supported by SFB 'Higher Invariants'.
\section{Preliminaries}\label{sec_prelim}

Throughout the article we assume that the reader is familiar with the basic facts about the solution of the Yamabe problem on closed manifolds by Trudinger, Aubin, Schoen and Yau. There are many beautifully written introductions in the literature, e.g. \cite{LP}, \cite{Hebey_book}. 

\subsection{Notations}\label{not_list}

In the article 
 a spin manifold always means a manifold admitting a spin structure together with a fixed choice of spin structure.
Spin structures can be defined for arbitrary oriented manifolds, but as soon as we have a Riemannian metric it yields a spin structure in the sense of $\Spin (n)$-principal bundles.

 For a Riemannian spin manifold $(M,g)$ we will always  write $S_M$ for the spinor bundle. In case the underlying manifold is fixed, we shortly write $S=S_M$. 
 
  The space of spinors, i.e., sections of $S$, is denoted by $\Gamma (S)$. The space of smooth compactly supported sections is called $C_c^\infty(M, S)$. The hermitian metric on fibers of $S$ is written as $\<.,.\>$, the corresponding norm as $|.|$.  We write $(.,.)_g$ for the $L^2$-product of spinors.
 
 We denote by $D\colon C_c^\infty(M, S) \to C_c^\infty(M, S)$ the Dirac operator on $(M,g)$. In case several manifolds or metrics are involved, we sometimes specify its affiliation, i.e., $D^g$, $D^M$ or $D^{M,g}$. Analogously we proceed for other operators and quantities.

The sphere $\mS^1$ carries two spin structures, one of them, the so-called bounding spin structure is obtained by restricting the unique spin structure on the two-dimensional 
disk to its boundary. The kernel of the Dirac operator for this spin structure is trivial. The sphere $\mS^1$ with the other spin structure represents the non-trivial 
spin-bordism class in dimension $1$.
In the article we will always assume that $\mS^1$ is equipped with the bounding spin structure, unless stated otherwise.

 A Riemannian manifold is of bounded geometry if it is complete, its injectivity radius is bounded from below and the curvature tensor and all derivatives are bounded.

 The ball around $x\in M$ with radius $\epsilon$ w.r.t. the metric $g$ on $M$ is written as $B^{M,g}_\epsilon(x)=B_\epsilon(x)\subset M$.

 In the article we need several Sobolev and Schauder spaces:  For $s\in [1,\infty]$ we write $\Vert . \Vert_{L^s(g)}$ for the $L^s$-norm on $(M,g)$. In case the underlying metric is clear from the context we abbreviate shortly by $\Vert. \Vert_s$.

 Let $H_k^s$ denote both the space of distributions on $M$ and the one of distributional sections in $S_M$ that have finite $H^s_k$norm given by 
\[ \Vert \phi\Vert_{H^s_k}^q=\sum_{i=0}^k \Vert (\nabla)^i \phi\Vert_{L^s}^s.\] Here $\nabla$ denotes the covariant derivative on $M$ and $S_M$, respectively, dependent whether $\phi$ is a distribution on $M$ and a distributional section in $S_M$, respectively.
$H_{s,\rm{loc}}^k$ means that any restriction of the distribution to a compact subset has to be in $H^s_k$ of that subset.

The space of $i$-times continuously differentiable functions on $M$ is denoted by $C^i(M)$, and $C^{i,\alpha}$ denotes the corresponding Schauder space for $\alpha\in (0,1]$.

\subsection{\texorpdfstring{The model spaces $\Mc^{m,k}$}{The model spaces Mc}}\label{mod_sp}

Let $0\leq k\leq m-1$ and $c\in [0,1]$. $(\Mc^{m,k}=\mH_c^{k+1} \times \mS^{m-k-1}, g_c=  g_{\mH_c^{k+1}}+\sigma^{m-k-1})$ where $\sigma^{m-k-1}$ denotes the standard metric on $\mS^{m-k-1}$ and $(\mH_c^{k+1}, g_{\mH_c^{k+1}})$ is the rescaled hyperbolic space with scalar curvature $-c^2k(k+1)$ if $c\in (0,1]$ and the Euclidean space if $c=0$. 

We introduce coordinates on $\mH_c^{k+1}$ by  equipping $\mR^{k+1}$
with the metric $g_{\mH_c^{k+1}}=\d r^2 +f(r)^2\sigma^k $ where 
\[ f_c(r):= \sinh_c(r):= \left\{ \begin{matrix} \frac{1}{c}\sinh(cr) &\quad
\mathrm{if\ }c\ne 0\\
r&\quad \mathrm{if\ }c=0. \end{matrix} \right. \]

The manifold $(\mathbb{M}_1^{m,k}=\mH^{k+1}\times \mS^{m-k-1}, g_1= g_{\mH^{k+1}} + \sigma^{m-k-1} =  \sinh^2t\, \sigma^{k+1} +\d t^2 + \sigma^{m-k-1})$ is conformal to $(\mS^m\setminus \mS^k, \sigma^m)$, \cite[Prop.~3.1]{ADH},
 \begin{equation*}
\mathfrak{u}\colon \mH^{k+1}\times \mS^{m-k-1} \to \mS^m\setminus \mS^k,\  g_1=f^2u^* \sigma^m\ {\rm where\ } f=f(t)=\cosh^2 t.\end{equation*}

\subsection{Regularity theory}

We recall the standard estimates:

\begin{thm}\label{est} Let $(M^m,g)$ be a Riemannian spin manifold of bounded geometry. Let $R>0$ be smaller than the injectivity radius of $M$, and let $r\in (0,R)$.
\begin{itemize}
 \item[(i)]\label{est_i} (Inner $L^s$-estimate, \cite[proof of Thm. 8.8]{GT}, spin version \cite[proof of Thm.~3.2.1 and~3.2.3]{Ammha}) 
  Let $\phi\in H_{1, \rm{loc}}^s$ be a solution of $D\phi=\psi$ for $\psi\in H_{k ,\rm{loc}}^s$. Then,
  there exists a constant $C=C(s, r, R)$ such that for all $x\in M$
 \[ \Vert \phi\Vert_{H_{k+1}^s(B_r(x))}\leq C (\Vert \phi\Vert_{L^s(B_{R}(x))} +\Vert \psi\Vert_{H_k^s(B_{R}(x))})\] 
 \item[(ii)]\label{est_ii} (Embedding into $C^{0,\gamma}$) Let $m< s$ and $0\leq \gamma\leq 1-\frac{m}{s}$. By the spin version of the proof of \cite[Sect.~7.8 (Thm.~7.26)]{GT} there exists a constant $C=C(s,r)$ such that 
 $H_1^s(B_{R}(x))$ is continuously embedded in $ C^{0,\gamma}(\overline{B_r(x)})$ for all $x\in M$.\\
 \item[(iii)]\label{est_iii} (Schauder estimates) \cite[Corollary~3.1.14]{Ammha} There is a constant $C=C(r,R,k)>0$ such that for $\alpha>0$, $\psi\in C^{k,\alpha}$ with $D\phi=\psi$  weakly it holds for all $x\in M$
\[ \Vert \phi\Vert_{C^{k+1,\alpha}(B_{r}(x))}\leq C( \Vert \phi\Vert_{C^k(B_{R}(x))}+\Vert \psi\Vert_{C^{k,\alpha}(B_{R}(x))}).\]
 \item[(iv)]\label{est_iv} (Sobolev Embedding into $L^p$, \cite[Thm.~7.26]{GT}) Let $k,\ell\in \mR$, $k\geq \ell$ and $s,t\in (1,\infty)$ with $k-(m/s)\geq \ell-(m/t)$,
then the restriction map $H^s_k(B_{R}(x), S)\to H^t_\ell(B_{r}(x), S)$ is continuous for all $x\in M$ and $r>0$. 
For fixed $R>r>0$ the operator norm of these restriction
maps can be chosen uniformly in $x$.
 \end{itemize}
\end{thm}

\subsection{\texorpdfstring{$L^s$-invertibility of Dirac operators}{Ls-invertibility of Dirac operators}}
For a complete manifold $M$, we define the norm $\Vert \phi\Vert_{\tilde{H}_1^s}:=\Vert \phi\Vert_s+\Vert D \phi\Vert_s$ for $1 \leq  s < \infty$. For $1\leq s<\infty$, let $\tilde{H}_1^s=\tilde{H}_1^s(M,S)$ be the completion of $C_c^\infty(M,S)$ w.r.t. the norm $\Vert \phi\Vert_{\tilde{H}_1^s}$. Then $D_s\colon \tilde{H}_1^s:=\dom\, D_s \subset L^s
\to L^s$ is a closed extension of the Dirac operator.  By \cite[Lem.~B.2]{ammann.grosse:p13a}  we have $(D_s)^*=D_{s^*}$ for $1<s<\infty$ and $s^{-1}+(s^*)^{-1}=1$.

Note that on manifolds of bounded geometry and $1<s<\infty$ the $H_1^s$-norm and the graph norms $\tilde{H}_1^s$ are equivalent, \cite[Lem.~A.2]{ammann.grosse:p13a}. 

General properties of the $L^s$-spectrum of Dirac operators can be found in \cite[App.~B]{ammann.grosse:p13a}. Here, we only cite the result on $L^s$-invertibility of our model spaces.

\begin{proposition}\cite[Thm.~1.1]{ammann.grosse:p13a}\label{inv_Mc}
Let $1\leq s\leq \infty$. The Dirac operator $D\colon L^s \to L^s$ on~$\Mc^{m,k}$ is $L^s$-invertible 
if $\lambda_1=\frac{m-k-1}{2}>ck\left| 1/s-1/2\right|$.       
\end{proposition}

\subsection{Yamabe type constants and Yamabe type invariants}

Let $(M^m,g)$ be a complete $m$-dimensional Riemannian spin manifold. 
By $\Delta^g$ we denote the Laplacian on $(M,g)$ and by $\scal^g$ its 
scalar curvature. Let $L^g=a\Delta^g+\scal^g$ be the conformal Laplacian 
where $a=4\frac{m-1}{m-2}$. We recall the following definitions
\begin{figure}
 \begin{tabular}{c|c|c|c|c}
 $a$ & $p$ & $p^*$ & $q$ & $q^*$\\[1mm]
 \hline
  &&&&\\[-3mm]
  $4\frac{m-1}{m-2}$ & $\frac{2m}{m-2}$ & $\frac{2m}{m+2}$ & $\frac{2m}{m-1}$ & $\frac{2m}{m+1}$  
 \end{tabular}
 \caption{Some constants in the article for $m$-dimensional manifolds}\end{figure}

\begin{defi}

{\bf Functionals}

\[\cF(v):=\frac{\int_M vL^gv\,\vo_g}{\Vert
v\Vert_{L^\frac{2m}{m-2}(g)}^2}, \quad \cFs(\phi):=\frac{\Vert D^g\phi\Vert_{L^\frac{2m}{m+1}(g)}^2}{(D^g\phi,\phi)_g}\]

For further use we define for the rest of the paper $p:=\frac{2m}{m-2}$ and $q:=\frac{2m}{m+1}$. The corresponding Euler-Lagrange equations for normalized solutions are the so-called Yamabe equation \cite{LP} 

\[ L^gv= \mu v^{\frac{m+2}{m-2}},\qquad \Vert v\Vert_{L^{p=\frac{2m}{m-2}}}=1\]

and the spinorial brother \cite{Amm03a}

\[ D^g\phi= \lambda |\phi|^{\frac{2}{m-1}}\phi,\qquad \Vert \phi\Vert_{L^{q=\frac{2m}{m-1}}}=1.\]

{\bf Yamabe type constants defined by compactly supported test functions.}

\begin{align*}
\Q(M,g):=& \inf \left\{ \cF(v) \ \Big|\ v\in
C_c^\infty(M,S)\setminus\{0\} \right\},\nonumber\\ 
\lm(M,g):= & \inf \left\{ \cFs(\phi)\ \Big|\ \phi\in C_c^\infty(M,S), (D^g\phi,\phi)_g>0 \right\}   
\end{align*}

{\bf Yamabe type constants defined over solutions.}

$\wQ(M,g)$ is the Yamabe invariant ``defined over the solutions'', i.e., 
\begin{align*}\wQ(M,g):=\inf\{\mu_v\,|\,v\in \Omega^{(1)}(M,g)\}
\end{align*}
where  $\Omega^{(1)}(M,g)$ is the set of all nonnegative functions
$v\in C^2(M)\cap  L^{\infty}(M)\cap  L^2(M,g)$
satisfying $L^g v = \mu_v v^{p-1}$ 
for a real number $\mu_v$ and with $\|v\|_{L^p(M,g)}=1$ ($p=\frac{2m}{m-2}$ as always).

Analogously, we introduce a quantity corresponding to $\lm(M,g)$  defined using the solutions of the Euler-Lagrange equation of $\mathcal{F}^{\rm spin}$:

\begin{align}{\wlm}(M,g):= \inf \{\lambda\in (0,\infty)\ |\ &\exists \phi\in L^\infty\cap L^2\cap C^1:\nonumber\\& 0<\Vert \phi\Vert_{L^\frac{2m}{m-1}(M,g)}\leq 1,\ D^g\phi=\lambda|\phi|^\frac{2}{m-1}\phi \}. \label{def_solLa} \end{align}

We will see in the next section why these different quantities are geometrically relevant.\\

{\bf Renormalized spinorial invariants}
We also introduce renormalized versions
of $\lm$ and $\wlm$:
\begin{align}\label{ren_inv}\Qs(M,g)=4\frac{m-1}m \lm(M,g)^2,\ \wQs(M,g)=4\frac{m-1}m \wlm(M,g)^2.\end{align}
 This renormalization will make simpler.\\

{\bf Yamabe type invariants for compact manifolds.}
Now we define the smooth Yamabe invariant $\si(M)$ as 
\[  \si(M):=\sup_{g} \Q(M,g)\]

where the supremum runs over all Riemannian metrics on $M$.  Thus, $\si$ only depends on the diffeomorphism type of $M$. Note that the smooth Yamabe invariant is positive if and only if $M$ admits a metric of positive scalar curvature.

A similar spinorial Yamabe invariant $\tau^+(M)$ was introduced in \cite{Ammtau, Amm04t}.  It is

\[ \tau^+(M):=\left\{ \begin{matrix}
                      \sup\nolimits_{g\in R^{\rm inv}(M)} \lm(M,g ) &  \text{if\ }R^{\rm inv}(M)\neq \varnothing\\
                     0\hfill &  \text{if\ }R^{\rm inv}(M)= \varnothing,
                      \end{matrix}
\right.\]
where $R^{\rm inv}(M)$ is the set of Riemannian metrics on $M$ such that $D^g$ is invertible. The definition of $\tau^+$ is slightly different to the original one in \cite{Ammtau, Amm04t}, but obviously equivalent. Note that for connected closed manifolds one knows from \cite{ammann.dahl.humbert:09} that $R^{\rm inv}(M)\neq \varnothing$ if and only if the index of $M^m$ in ${\rm KO}_m$ vanishes.
Thus, $\tau^+(M)$ is  positive if and only if this index vanishes. 
The invariant $\tau^+(M)$ only depends on the diffeomorphism type of $M$ and its spin structure. 

These Yamabe type invariants will be considered in this article only in the case that $M$ is compact. In this case  the solution of the classical Yamabe problem \cite{LP} implies $\wQ(M,g)=\Q(M,g)$, and similar results in the spin case \cite{Amm03a} implies   
${\wlm}(M,g)=\lm(M,g)$. Thus, we also see 
\[  \si(M)=\sup_{g} \wQ(M,g), \qquad \tau^+(M)=\left\{ \begin{matrix}
                      \sup\nolimits_{g\in R^{\rm inv}(M)} \wlm(M,g ) &  \text{if\ }R^{\rm inv}(M)\neq \varnothing\\
                     0\hfill &  \text{if\ }R^{\rm inv}(M)= \varnothing.
                      \end{matrix}
\right.\]
Similar to above we also define (for $M$ compact) a renormalized version
\[ \sis(M):=4\frac{m-1}m\tau^+(M)^2.\]

We want to remark that $\si(M)$ was considered for non-compact manifolds
in \cite{grosse.nardmann:p12}. 
\end{defi}

{\bf The $\Lambda$-invariants.}
We define 
\[ \wLa_{m,k}:= \inf_{c\in [0,1]} \wQ (\Mc^{m,k})\text{\ and\ } \sLa_{m,k}:= \inf_{c\in [0,1]} Q^* (\Mc^{m,k}).\]

These invariants are important because of their relation to the invariant $\La_{m,k}$ that contributes to the Surgery Theorem ~\ref{surgthm}. We have $\La_{m,k}=\wLa_{m,k}$ unless $m=k-3\geq  7$ or $m=k-2$ \cite[Thm.~3.1 and proof of Corollary~3.2]{ammann.dahl.humbert:p11b}. The idea behind the notation is that the invariant with $*$ is the infimum of a functional, the invariant with $\sim$ is defined using solutions of the Euler-Lagrange equation and the invariant without such decoration is the invariant in the surgery theorem. We know from \cite[Thm.~3.3]{ammann.dahl.humbert:p11b} that all these invariants are positive for $0\leq k\leq m-3$.

In the spinorial case we define similarly 
\[ \wLas_{m,k}=\inf_{c\in [0,1]} \wQs(\Mc^{m,k})\ \text{and}\ \sLas_{m,k}:=  \inf_{c\in [0,1]} \Qs (\Mc^{m,k}).\] It is known from \cite[Thm.~1.1]{ADH09} that $\wLas_{m,k}>0$ for all $0\leq k\leq m-2$.
The invariant $\Las_{m,k}$ in the Spinorial Surgery Theorem~\ref{surgspinthm} can be chosen to be $\wLas_{m,k}$ for all $0\leq k\leq m-2$, \cite[Cor.~1.4]{ADH09}. We introduce the notation  $\Las_{m,k}:=\wLas_{m,k}$ to make the presentation analogous to the non-spin case.

One of the main goals of this article is to search for relations between these five possibly different $\Lambda$-invariants.

\begin{remark}\label{inv_sphere}
 The $Q$-invariants for the spheres play a special role. We collect the main properties:  For all manifolds $(M^m,g)$ it holds $\Q (M^m,g)\leq \Q (\mS^m)$. 
 If $M$ is spin, then $\Qs (M^m,g)\leq \Qs (\mS^m)$.
  The invariant $\Q(\mS^m)$ is attained by a constant test function $v$ such that $\Vert v\Vert_{\frac{2m}{m-2}}=1$. Thus, $Lv=m(m-1)v=\Q(\mS^m) v^{\frac{m+2}{m-2}}$ and $\Q(\mS^m)=m(m-1)\rm{vol}(\mS^m)^{\frac{2}{m}}$. 
  The invariant $\lm(\mS^m)$ is attained by a Killing spinor $\phi$ to the Killing constant $-\frac{1}{2}$. Note that the normalization in \eqref{ren_inv} is chosen such that $\Qs(\mS^m)=\Q(\mS^m)$. Since $\mS^m$ is closed $\Q(\mS^m)=\wQ(\mS^m)$ and $\Qs(\mS^m)=\wQs(\mS^m)$. 
 
  Moreover, for $(M^m,g)$ not locally conformally flat and $m\geq 6$ Aubin showed, see \cite[p.292]{Aubin}, $\Q (M^m, g)< \Q (\mS^m)$. 
\end{remark}

\subsection{Surgery-monotonicity for Yamabe type invariants below thresholds}\label{surg_theorems}

In order to define the constant $\La_{m,k}$ mentioned above we set
\[Q^{(2)}(M,g):=\inf\{\mu_u\,|\,u\in \Omega^{(2)}(M,g)\}
\]
where  $\Omega^{(2)}(M,g)$ is the set of all nonnegative functions
$u\in C^2(M)\cap  L^{\infty}(M)$ satisfying $L^g u = \mu_u u^{p-1}$ 
for a nonnegative real number $\mu_u$, $\|u\|_{L^p(M,g)}=1$, where $p=\frac{2m}{m-2}$ as always, and $\mu_u\Vert u\Vert_{L^\infty}^{\frac{4}{m-2}}\geq \frac{(m-k-2)^2(m-1)}{8(m-2)}$. Then, we set, cf. \cite[Sect.~3]{ADH}, \cite[Sect.~2.6]{ammann.dahl.humbert:p11b}, 
\[\La_{m,k}=\min\{ \wLa_{m,k}, \inf_{c\in[0,1]} Q^{(2)}(\Mc^{m,k}) \}.\] 
It follows from \cite[Thm.~3.1 and below]{ammann.dahl.humbert:p11b} that in case $k=m-3\leq 6$ or $k\leq m-4$ we already have $\La_{m,k}=\wLa_{m,k}$.

\begin{theorem}[Surgery-monotonicity for the Yamabe invariant, {\cite[Cor.~1.4]{ADH}}]\label{surgthm} 
Assume that~$N^m$ is a closed Riemannian  manifold that is obtained from $M^m$ by a surgery of codimension $m-k\geq 3$.
Then \[ \si(N)\geq \min\{ \si(M), \La_{m,k}\}.\] 
\end{theorem}

Note that a surgery from $M$ to $N$ is called spin preserving if the spin structures on $M$ and $N$ extend to a spin structure on the corresponding bordism. 
In particular this implies that the spin structures  on $M$ and $N$ coincide 
outside the region of surgery.

\begin{theorem}[Surgery-monotonicity for the spinorial Yamabe invariant, {\cite[Cor.~1.4]{ADH09}}]\label{surgspinthm} 
Assume that $N^m$ is a closed Riemannian spin manifold that is obtained from $M^m$ by a spin-preserving surgery of codimension $m-k\geq 2$.
Then \[ \sis(N)\geq \min\{ \sis(M), \Las_{m,k}\}.\] 
In the case $k=m-2$ the sphere $\mS^1$ carries the bounding spin structure, as explained in the Notations~\ref{not_list}.
\end{theorem}

\begin{figure}[h]
\hskip-0.2cm
\begin{tabular}{c|l}
Label & Meaning  \\[3mm] 
\hline
&\\[-3mm]
no $\rm spin$ label & associated to the Yamabe problem\\
$\rm spin$ & associated to the spinorial Yamabe problem for the Dirac operator\\
$*$ & defined as a variational problem\\
$\sim$ & defined by solutions of the associated Euler-Lagrange equation\\
&(as a nonlinear eigenvalue problem)\\
no $*$ and no $\sim$ & appears in an associated surgery result\\
&\\[-3mm]
$Q$&  for a fixed Riemannian (spin) manifold\\
$\Lambda$ & $=\inf_{c\in[0,1]}$ of the $Q$-invariants (decorated with the same labels as  $\Lambda$) for\\ & the model spaces $\Mc$\\
$\sigma$ & the supremum of the $Q$-invariants (with the same labels) for a\\ & (spin) manifold over all conformal classes
\end{tabular}
\caption{Overview on the notations \label{tab_not}}
\end{figure}

Since there is whole zoo of different $Q$- and $\Lambda$-invariants, we summarize the logic of our notation in Figure~\ref{tab_not}.

\section{Overview on the results}\label{sec_results}

Many of the inequalities established in this article are summarized in Figure~\ref{fig.result}. For example, $\sLas_{m,k}=\wLas_{m,k}\geq \sLa_{m,k}$ for all $k\leq m-2$. Other inequalities hold under additional assumptions, e.g. in the case $k\leq m-4$  and in the case $k\leq m-3 \leq 3$ we have $\sLa_{m,k}\geq \wLa_{m,k}$. Thus, together with previously mentioned relations we obtain 

\begin{thm}\label{theo3.1}
In the case $k\leq m-4$  and in the case $k\leq m-3 \leq 3$
\[ \sLas_{m,k}=\wLas_{m,k}=\Las_{m,k}\geq  \sLa_{m,k}=\wLa_{m,k}=\La_{m,k}.\]
\end{thm}

\begin{figure}[h]
\centering
\savenotes
\begin{tikzpicture}
 \path[draw] (0,0) node {$\sLas_{m,k}$}  (0,-0.4) to[bend right] (5,-.4)  (5,0) node {$\wLas_{m,k}$};
 \path[draw]  (0,0.4) to[bend left] (5,.4) ;
 \path[draw] (2.5,1.4) node {$\leq$ \tiny{(for $k\leq m-2$, Cor.~\ref{Qs_leq_wQs})}};
 \path[draw] (2.5,-.4) node {$\geq$ \tiny{(for $k\leq m-2$, Prop.~\ref{wLas_leq_Las})}};

 \path[draw] (0,-5) node {$\sLa_{m,k}$}  (0,-5.4) to[bend right] (5,-5.4)  (5,-5) node {$\wLa_{m,k}$};
 \path[draw]  (0,-4.6) to[bend left] (5,-4.6) ;
 \path[draw] (2.5,-4.4) node {$\leq$ \tiny{(Cor.~\ref{La_leq_wLa})}};
 \path[draw] (2.5,-6.4) node {$\geq$ \tiny{(for $k\leq m-4$ or $k\leq m-3\leq 3$, Prop.~\ref{wLas_leq_Las})}};

    \node[anchor=east] at (0,-4.6) (text) {};
  \node[anchor=west] at (5,-0.4) (description) {};
  \draw (description) .. controls ([yshift=-4cm] description) and ([yshift=4cm] text) .. (text);
 \path[draw] (2.5,-2) node {$\leq$ \tiny{(for $k\leq m-2$, Cor.~\ref{Q_leq_wQs})}};

 \path[draw] (7,0) node {$\Qs (\Mc^{m,k})$}  (7,-0.4) to[bend right] (12,-.4)  (12,0) node {$\wQs (\Mc^{m,k})$};
 \path[draw]  (7,0.4) to[bend left] (12,.4) ;
 \path[draw] (9.5,1.4) node {$\leq$  \tiny{(for $k\leq m-2$,
Cor.~\ref{Qs_leq_wQs})}  };
 \path[draw] (9.5,-.4) node {$\geq$ \tiny{(Cor.~\ref{wQs_leq_Qs}$^{[1]}$)}};
 
 \path[draw] (7,-5) node {$\Q (\Mc^{m,k})$}  (7,-5.4) to[bend right] (12,-5.4)  (12,-5) node {$\wQ (\Mc^{m,k})$};
 \path[draw]  (7,-4.6) to[bend left] (12,-4.6) ;
 \path[draw] (9.5,-4.4)  node {$\leq$ \tiny{(Cor.~\ref{La_leq_wLa})}};
 \path[draw] (9.5,-6.4) node {$\geq$ \tiny{(Cor.~\ref{wQ_leq_Q}$^{[2]}$)}};

     \node[anchor=east] at (7,-4.6) (text) {};
  \node[anchor=west] at (12,-0.4) (description) {};
  \draw (description) .. controls ([yshift=-4cm] description) and ([yshift=4cm] text) .. (text);
 \path[draw] (9.5,-2) node {$\leq$ \tiny{(for $k\leq m-2$, Cor.~\ref{Q_leq_wQs})} };

  \end{tikzpicture}
\spewnotes
\caption{Summary of the results for the $Q$-invariants of the model spaces (right) and the corresponding $\Lambda$-invariants (left).\newline
  $\quad [1]$: for $(m-k-1)^2>c^2k(k+1)$ and $\wQs(\Mc^{m,k})< \Qs(\mS^m)$ \newline
  $\quad [2]$: for $(m-k-1)(m-k-2)>c^2k(k+1)$, $c\in [0,1)$ or for $k\leq m-3$, $c=1$\label{fig.result}
}
\end{figure}

Thus, in most of the cases the  inequality $\Las_{m,k}\geq \La_{m,k}$ conjectured in the introduction holds. Together with the explicit positive lower bounds for $\La_{m,k}$ in \cite{ammann.dahl.humbert:p12, ammann.dahl.humbert:p11b}, we then obtain explicit positive lower bounds for $\Las_{m,k}$. Theorem~\ref{theo3.1} does not provide for $\Las_{m,m-3}$ for $m>6$. But nevertheless our techniques also allow to obtain explicit positive lower bounds for  $\Las_{m,m-3}$ for $m>6$, see Section~\ref{km3}.

The right hand side of Figure~\ref{fig.result} gives relations between the $Q$-invariants of the model spaces. 
Some of them require additional assumptions which are given as footnotes. 
The parameter $c$ ranges in the interval $[0,1]$. 
However, the case $c=1$ is very special as then $\mathbb{M}_1^{m,k}$ is 
conformal to a subset of $\mS^m$ which allows much stronger statements. 
This is summarized in Section~\ref{c=1}. Another special case is $k=m-1$. 
These invariants do not have similar geometric applications. But for the sake of completeness we summarize in Section~\ref{km1}.

As already mentioned in the introduction, the explicit positive lower bounds 
for $\Las_{m,k}$, $ k\leq m-3$ lead to bordism invariant. As we only want to give an overview here, many proofs will be given later, i.e. in 
section~\ref{sec.bordism.arg}.

Let $m\geq 5$. We set 
  $$\Las_m:=\min \{\Las_{m,2},\Las_{m,3},\ldots,\Las_{m,m-3}\}.$$
{}From Theorem~\ref{theo3.1}, Section~\ref{km3}, and results in  \cite{ammann.dahl.humbert:p11b} and \cite{ammann.dahl.humbert:p12} we obtain explicit positive lower bounds for 
$\Las_m$, summarized in Table~\ref{tab:Las-m} for low dimensions.

\begin{table}
  \centering
  \begin{tabular}{c||c|c|c|c|c|c|c|c|c} 
    $m$ & 5 & 6 &  7& 8& 9& 10 & 11&12&13\\
    \hline
    $\Q(\mS^m)$  & 80.0& 96.3 &
    113.5 & 130.7 & 147.9 & 165.0& 182.2 &199.3 &216.4\\ 
    \hline
    $\Las_m\geq $  & 45.1 & 50.0 &
    65.2 & 78.7 & 91.8 & 104.9& 118.1 &131.5& 145.0 
\\
    \hline
    $Q^*(\mathbb{H}P^2\times \mathbb{R}^{m-8})\geq$  & ---& ---&---& 121.4& 138.5 & 97.3 & 135.9 & 158.7 & 178.0
  \end{tabular}
\\[1mm]
  \caption{Some explicit lower bounds for $\Las_m$: The values are rounded -- 
$\Q(\mS^m)$ is rounded to the nearest multiple of $1/10$ and the lower bounds for $\Las_m$ and   $Q^*(\mathbb{H}P^2\times \mathbb{R}^{m-8})$ are always 
rounded down.  Note that $Q^*(\mathbb{H}P^2)=121.4967...$ is attained by the canonical metric on $\mathbb{H}P^2$, due to Obata's Theorem. 
}
  \label{tab:Las-m}
\end{table}



Using standard techniques from bordism theory (see Section~\ref{sec.bordism.arg} for details) one obtains several conclusions:

\begin{proposition}\label{sis.bound}
Let $M$ be an $m$-dimensional closed connected, simply connected spin manifold, $\alpha(M)= 0$. 
If 
$5\leq m \leq 7$, then 
  \begin{align*} \sis(M)\geq \Las_m.\end{align*}

For $m\geq 11$ or $m=8$ we have 
  \begin{align*}\sis(M)\geq \min\{\Las_m, \Q(\mathbb{H}P^2\times \mathbb{R}^{m-8})\}.\end{align*}
  For $m=9,10$ we have 
  \begin{align*}\sis(M)\geq \min\{\Las_{m,1}, \Las_m,  \Q(\mathbb{H}P^2\times \mathbb{R}^{m-8})\}.\end{align*}
\end{proposition}

Note that by definition $\alpha(M)\neq 0$ implies that there are no invertible
Dirac operators, thus by definition $\sis(M)=0$.

We conjecture  $\Las_m\leq \Q(\mathbb{H}P^2\times \mathbb{R}^{m-8})$ for all $m\geq 11$, which would imply $ \sis(M)\geq \Las_m$ for all closed simply connected spin manifolds $M$ with dimension $m\geq 5$, $m\neq 9,10$.

A similar bound also exists for non-simply connected manifolds, namely in this case for $m\neq 9,10$    
  $$\sis(M)\geq \min\{\Las_m,\Las_{m,m-2},\Q(\mathbb{H}P^2\times \mathbb{R}^{m-8})\}>0,$$
  and for $m=9,10$
  $$\sis(M)\geq \min\{\Las_{m,1}, \Las_m,\Las_{m,m-2}, \Q(\mathbb{H}P^2\times \mathbb{R}^{m-8})\}>0,$$
however this positive lower bound is not explicit as no explicit lower bound 
for $\Las_{m,m-2}$ is currently available. Numerical calculations and 
some further assumptions indicate that $\Las_{m,m-2}<\Las_m$. 

\begin{proposition}\label{sis.bound.fund}
Assume that $M$ is an $m$-dimensional closed connected spin manifold with $m\geq 5$.
We consider the bordism
groups $\Omega_m^{\rm spin}(B\Gamma)$, $\Gamma:=\pi_1(M)$ where the boundaries and the bordisms are spin manifolds together with  
maps to $B\Gamma$. Let $c_M\colon M\to B\Gamma$ be a classifying map of the universal covering of $M$, i.e.,\ the map which induces 
an isomorphism from $\pi_1(M)$ to $\Gamma=\pi_1(B\Gamma)$. Let $[N,f]=[M,c_M] \in  \Omega_m^{\rm spin}(B\Gamma)$, and let $N$ be connected. 
Then
  $$\sis(N)\geq \min\{\sis (M), \Las_m, \Las_{m,m-2}\}.$$
If $N$ is connected and if $f$ induces an isomorphism from $\pi_1(N)$ to $\Gamma$, then
  \begin{align}\sis(N)\geq \min\{\sis (M), \Las_m\}.\label{sis_in}\end{align}
  \end{proposition}

Note that every class in $\Omega_m^{\rm spin}(B\Gamma)\to \mathbb{R}$ can be written as $(M,c_M)$.

By applying \eqref{sis_in} twice, it follows from Proposition~\ref{sis.bound.fund} that there is a well-defined map
$s^{\rm spin}\colon \Omega_m^{\rm spin}\to \mathbb{R}$ such that for all connected, simply connected spin manifolds $M$
\[ s^{\rm spin} ([M])=\min \{ \sis(M), \Las_{m}\}. \]

Thus if $M$ is a connected spin manifold, we have
\[ \sis(M)\geq \min \{ s^{\rm spin} ([M]), \Las_{m,m-2}\}. \]

It follows from standard arguments of surgery theorem that
\[ s^{\rm spin} ([M]+[N])\geq \min\{ s^{\rm spin} ([M]), s^{\rm spin} ([N])\}.\]

Thus, 
\[ \Omega_m^{{\rm spin}, > t}:=\{ [M]\in \Omega_m^{\rm spin} \ |\ s^{\rm spin}([M])>t\}\]
is a subgroup of $ \Omega_m^{\rm spin}$.
For example,  $\Omega_m^{{\rm spin}, > 0}$ is the kernel of the index map $\alpha\colon \Omega_m^{{\rm spin}}\to KO_m$.

Similarly as above, for an arbitrary finitely presented group $\Gamma$ we obtain a well-defined map 
$s_\Gamma^{\rm spin}\colon \Omega_m^{\rm spin}(B\Gamma)\to \mathbb{R}$ as follows: For every $[M,f]\in \Omega_m^{\rm spin}(B\Gamma)$ where $M$ is connected and $f$ induces an isomorphism from $\pi_1(M)$ to $\Gamma$ we have
  $$s_\Gamma^{\rm spin}([M,f])= \min \{ \sis (M), \Las_{m,1} , \Las_m\}.$$
  
In this case the minimum includes the constants $\Las_{m,1}$ since it is required to show that 
\[ s^{\rm spin}_\Gamma ([M,f]+[N,g])\geq \min\{ s_\Gamma^{\rm spin} ([M,f]), s_\Gamma^{\rm spin} ([N,g])\}.\]
Then, analogously as above, 
\[ \Omega_m^{{\rm spin}, > t}(B\Gamma):=\{ [M,f]\in \Omega_m^{\rm spin}(B\Gamma) \ |\ s_\Gamma^{\rm spin}([M,f])>t\}\]
is a subgroup of $ \Omega_m^{\rm spin}(B\Gamma)$.\\

Assume that there is a closed simply connected spin manifold $M$ of dimension $m \geq 5$ 
with $\sis(M)< \Las_m$.
For such manifolds one would have: 
If $N$ is a simply connected closed spin manifold spin-bordant to $M$, 
then $\sis(N)= \sis(M)$. An advantage of this bordism result is that we have explicit positive lower bounds for $\Las_m$, in contrast to a similar result for the classical Yamabe invariant.

As a consequence by the Hijazi inequality we have 
$\si(N)\leq \sis(M)$, i.e., $\sis(M)$ is an upper bound for the Yamabe invariant for all simply connected manifolds in $[M]$.

This question is related to the open problem whether there is a manifold in dimension $m\geq 5$ with Yamabe invariant different from $0$ and $\si(\mS^m)$.
If one finds an $M$ as above, all simply connected manifolds in the spin bordism class of $M$ would have a Yamabe invariant in $(0, \si(\mS^m))$. \\

Many of the statements on the right hand side of Figure~\ref{fig.result} are 
still valid if one replaces the model spaces by arbitrary manifolds of bounded geometry, 
see Sections~\ref{sec_cut} and~\ref{sec_hijazi}. 
The inequalities in Section~\ref{sec_hijazi} are noncompact versions of the Hijazi inequality which is of central importance of our article. 
The reader should be aware that there are different ways to generalize from the compact to the noncompact setting. 
We have positive and negative results for the generalization of the Hijazi inequality to the noncompact setting, see  Section~\ref{sec_hijazi}.  
Our investigations also need regularity statements for the Euler-Lagrange equation of the spinorial functional. 
For this purpose we have included Section~\ref{sec_reg} which might be of independent interest and which goes beyond the requirements of the following sections.

\section{Improvements of regularity for the Dirac Euler-Lagrange equation}\label{sec_reg}

Let $(M^m,g)$ be a Riemannian spin manifold of bounded geometry. In this section, we consider a spinor $\phi\in L^q$ and $\phi\in L^s_{\rm loc}$ for an $s>q$ that fulfills 
\begin{align} D\phi=\lambda|\phi|^{q-2}\phi \quad\text{\ weakly},\label{eq_we}
\end{align}
i.e., in the distributional sense, where as always $q=\frac{2m}{m-1}$.  Note that from  $\phi\in L^s_{\rm loc}$ for an $s>q$ it follows with the methods of \cite[Thm.~5.2]{Ammha} that $\phi$ is $C^{1,\alpha}$ for all $\alpha\in (0,1)$. We omit the proof of this local statement since the proof is completely analogous as in \cite{Ammha}. Furthermore, we will only use the fact that $\phi$ is continuous which is part of the assumptions in the applications of this subsection.

We want to further examine the regularity of $\phi$.
First, we will show that $\phi\in L^\infty$. For that we need the following auxiliary lemma.

\begin{lemma} \label{seq-bddratio} Fix $\beta, R, \delta >0$. Let $\phi\in \Gamma(S_M)$ be continuous with $\Vert \phi\Vert_{L^\infty}=\infty$.  

Then there is a sequence $(x_i)_{i\in \mN}$ in $M$ with $|\phi(x_i)|\geq i$ and
\[ |\phi(x_i)|^{-1} \Vert \phi\Vert_{L^\infty(B_{i}^R)}\leq 1+\delta\]
where $B_{i}^R:= B_{R\,|\phi(x_i)|^{-1/\beta}} (x_i)$.
\end{lemma}

\begin{proof}  Let $d(.,.)$ denote the distance in $(M,g)$, and fix $R,\delta >0$. We prove the claim by contradiction: Assume that there is a constant $C>0$ such that for all $x\in M$ with $|\phi(x)|\geq C$  there is $y_x$ with $d(x,y_x)< R\,|\phi(x)|^{-1/\beta} $ and $|\phi(y_x)|>(1+\delta)\,|\phi(x)|$. Then, we define a sequence $x_i$ recursively by choosing $x_0\in M$ with $|\phi (x_0)|\geq C$ and $x_{i+1}=y_{x_i}$ for all $i\geq 0$. Then, $|\phi(x_i)|\geq (1+\delta)^{i}|\phi(x_0)|\geq  (1+\delta)^{i}C\to \infty $ as $i\to \infty$. But, 
  \[d(x_i, x_0)\leq \sum_{j=0}^{i-1} d(x_{j+1},x_{j})\leq \sum_{j=0}^{i-1}  R\,|\phi(x_j)|^{-\frac{1}{\beta}} \leq R C^{-\frac{1}{\beta}}  \sum_{j=0}^{i-1}  (1+\delta)^{-\frac{j}{\beta}} \leq  \frac{ R C^{-\frac{1}{\beta}}}{1-(1+\delta)^{-\frac{1}{\beta}}}<\infty\]
  which then contradicts the continuity of $\phi$.
\end{proof}

\begin{lemma}\label{bdd-sol}  Let $(M^m,g)$ be of bounded geometry.
 Let $\phi\in L^q\cap C^0$  be a weak solution of~\eqref{eq_we}. Then $\phi\in L^\infty$.
 \end{lemma}

\begin{proof}   We assume the contrary, i.e.,  $\Vert \phi\Vert_\infty=\infty$.  We fix $\beta:=2(q-2)$, $R$ smaller than the injectivity radius, and some $\delta>0$. 
Then applying Lemma~\ref{seq-bddratio}  there is a sequence of points $(x_i)_{i\in \mN}$ in $M$ with $|\phi(x_i)|\geq i$ and $\Vert \phi\Vert_{L^\infty(B_i^R)}\leq (1+\delta)|\phi (x_i) |$. After passing to a subsequence,  every compactum only contains a finite number of $x_i$. We thus assume that all $B_i^R$ are pairwise disjoint since this can always be achieved by passing to a further subsequence. We consider the charts for $B_i^R$ given by rescaled exponential maps
  \[ u_i\colon B_{r_i}(0)\subset \mR^m\to B_i^R,\ v\mapsto \exp_{x_i} (\delta_i v)\]
  where $m_i:=|\phi(x_i)|$, $\delta_{i}:=m_i^{-\frac{1}{q-2}}$ and $r_i:=\delta_i^{-1} R\,|\phi (x_i)|^{-\frac{1}{\beta}}= R\,|\phi (x_i)|^{\frac{1}{2(q-2)}}$.

  Note that $m_i=|\phi(x_i)|\geq i\to \infty$ and, hence, $\delta_i\to 0$ and $r_i\to \infty$ as  $i\to \infty$.
  The map $u_i$ induces a map on the frame bundles which lifts to the spinor bundles, for details see \cite{BG_92}.
For simplicity, we denote this lift also by $u_i$, and set  $\psi^i:= m_i^{-1} u_i^*\phi$. Then $\psi^i$ is a spinor on $B_{r_i}(0)$, $|\psi^i(0)|=1$ and $\Vert \psi^i\Vert_{L^\infty(B_{r_i}(0))}\leq 1+\delta$.

Using the comparison of the Dirac operator with the one on the Euclidean space  \cite[Sect.~3 and 4]{AGH03}, we obtain from $D\phi=\lambda|\phi|^{q-2}\phi$ that

\[ D^{\mR^m}\psi^i +\frac{1}{4} \sum_{\alpha\beta\gamma}\tilde{\Gamma}_{\alpha\beta}^\gamma e_\alpha\cdot e_\beta\cdot e_\gamma\cdot \psi^i +\sum_{\alpha\beta} (b_\alpha^\beta-\delta_\alpha^\beta)e_\alpha\cdot \nabla_{e_\beta} \psi^i
 =\lambda|\psi^i|^{q-2}\psi^i
\]
where $\delta_\alpha^\beta$ denotes the Kronecker symbol, $e_\alpha$ is the standard orthonormal frame on $\mR^m$ and
\begin{align*}
b_{\alpha}^\beta=&\delta_\alpha^\beta-\frac{1}{6}\delta_i^2R^i_{\alpha \lambda\mu\beta} x^\lambda x^\mu+O(\delta_i^3|x|^3)\to \delta_\alpha^\beta\\
\tilde{\Gamma}_{\alpha\beta}^\gamma=&\partial_\alpha b_\beta^\gamma -\frac{1}{3}\delta_i(R^i_{\alpha\gamma\lambda\beta}+R^i_{\alpha\lambda\gamma\beta})x^\lambda +O(\delta_i^2|x|^2)\to 0 
\end{align*}
as $\delta_i\to 0$, $i\to\infty$. Here, 
$R^i_{\alpha \lambda\mu\beta}= g_{x_i}([\na_{\partial_\beta},\na_{\partial_\mu}]\pa_\alpha-\na_{[\pa_\beta,\pa_\mu]}\pa_\alpha,\pa_\lambda)$ is the Riemannian curvature tensor of $g$ at $x_i$.

Let $K_j$, $j=0,1,2,3$ be compact subsets of $\mathbb{R}^m$ with $K_{j+1}\subset {\rm interior}(K_j)$ for $j=0,1,2$, 
and let $i_0$ be big enough such that $K_0\subset B_{r_{i_0}}(0)$. 
Since $\psi^i$ is bounded on $B_{r_i}(0)$ for $i\geq i_0$, the inner $L^s$-estimate in 
Theorem~\ref{est_i} shows that for each $s$ the $\psi^i$'s are uniformly 
bounded in $H_1^s(K_0)$. 
Thus, after passing to a subsequence $\psi^i\to \psi$ weakly in $H_1^s(K_0)$. 
The restriction map $H_1^s(K_0)\to C^{0,\gamma}(K_1)$  is bounded because of   Theorem~\ref{est_ii}, hence
the $\psi^i$ are uniformly bounded also in $C^{0,\gamma}(K_1)$ for all 
$\gamma\in (0,1)$. In particular, $|\psi^i|^{q-2}\psi^i$ are uniformly bounded 
in $C^{0,\gamma}(K_1)$. Thus, by the Schauder estimate 
(see Theorem~\ref{est_iii}) we obtain $\psi\in C^{1,\gamma}(K_2)$ and, thus, by 
Arzela-Ascoli $ \psi^i\to \psi$ strongly in $C^1$ on $K_3$, 
after passing to a subsequence. We apply this construction to $K_3:=B_k(0)$ and construct a diagonal subsequence  for $k\to \infty$. This subsequence converges locally in $C^1$ to a spinor $\psi$ on $\mR^m$ with 
$|\psi(0)|=1$, $\Vert\psi\Vert_{L^\infty}\leq 1+\delta$ and $D^{\mR^m}\psi=\lambda |\psi|^{q-2}\psi$.

We write $b_i\vo_{\mR^m}= u_i^*\vo_g$, where $b_i:=\sqrt{\det (u_i^*g)}\to 1$ as 
$i\to \infty$ in $C^1$ on each compact subset of $\R^m$. As the balls $B_i^R$ are disjoint, $\phi\in L^q$ implies that  $\int_{B_i^R} |\phi|^q \,\vo_g  \to 0$ as $ i\to\infty$.
Thus, we get for all compacta $\tilde{K}$ and sufficiently large $i$

\begin{align*}
 \int_{\tilde{K}}|\psi^i|^q\, \vo_{\mR^n} \leq 1.001 \int_{|x|< r_i}|\psi^i|^q b_i\,\vo_{\mR^n} = 1.001 \int_{B_i^R} |\phi|^q \,\vo_g  \to 0\quad\text{as\ } i\to\infty.
\end{align*}
Hence, $\Vert \psi\Vert_q=0$ which contradicts $\psi\in C^1$ and $|\psi(0)|=1$. Thus, $\phi\in L^\infty$. 
\end{proof}

\begin{lemma}\label{0-bdd-sol} Let $(M^m,g)$ be of bounded geometry.
 Let $\phi\in L^q\cap C^0$ fulfill weakly \eqref{eq_we}. 
 Then, $\lim_{x\to \infty} |\phi|=0$. 
 Moreover, $\phi\in C^{1,\gamma}$ for all $\gamma\in (0,1)$, $\lim_{x\to \infty} \Vert \phi\Vert_{C^{1,\gamma}(B_r(x))}=0$ for all $r>0$, and $\Vert \phi\Vert_{C^{1,\gamma}}<\infty$. In particular, 
$\phi$ is uniformly continuous.
\end{lemma}

\begin{proof} From Lemma~\ref{bdd-sol} we have $\phi\in L^\infty$. Fix $z\in M$ and $\delta>0$ to be smaller than the injectivity radius. Let $d(., .)$ denote the distance function on $(M,g)$. We prove the first claim by contradiction: We assume that there is constant $V > 0$ and a sequence $(x_i)_{i\in \mN}\subset M$ with $|\phi(x_i)|\geq V$, $|x_i|=d(x_i, z)\to \infty$ and $d(x_i,x_j)> 2\delta$.

Let $\epsilon \in (0,\frac{\delta}{2})$. 
Since $\phi\in L^\infty$ and $(M,g)$ has bounded geometry, we obtain by inner $L^s$-estimates that
\begin{align} \Vert \phi\Vert_{H_1^s(B_{\epsilon}(x_i))}&\leq C_\delta(s)(\Vert \phi\Vert_{L^s(B_{2\epsilon}(x_i))} + \Vert \lambda |\phi|^{q-2}\phi\Vert_{L^s(B_{2\epsilon}(x_i))})\nonumber \\&\leq CC_\delta(s)\vol(B_{2\epsilon}(x_i))^\frac{1}{s}\leq C'\label{H_1_s}\end{align}
where $C'$ does not depend on $i$.  

Fixing $s>m$ and using the Sobolev embedding $H_1^s(B_\epsilon(x_i)) \hookrightarrow C^{0,\gamma}(B_\epsilon(x_i))$ we get that $\Vert \phi\Vert_{C^{0,\gamma}(B_{\rho}(x_i))} \leq C''$ for some $\gamma\in (0,1)$, $\rho\in (0,\epsilon)$, and where $C''$ is independent on $i$.

With $\phi \in L^q$ we estimate
\[ \Vert \phi\Vert_{q}^q \geq \sum_i \Vert \phi\Vert_{L^q(B_{\rho}(x_i))}^q \geq K \sum_i \inf_{x\in B_{\rho}(x_i)} |\phi(x)|\]
where $K:=\inf_i \vol(B_{\rho} (x_i))$. Note that $K>0$ since $(M,g)$ has bounded geometry. Hence, $\inf_{x\in B_{\rho}(x_i)} |\phi(x)|\to 0$ as $i\to \infty$.
But on the other hand on each ball $B_{\rho}(x_i)$ we have  for all $x,y\in B_{\rho}(x_i)$ that $|\phi(x)-\phi(y)|\leq C''|x-y|^\gamma\leq C'' {\rho}^\gamma$. 
Thus 
  $$V\leq \limsup_{i\to\infty} |\phi(x_i)| \leq C'' {\rho}^\gamma.$$
By choosing $\rho$ small enough we obtain a contradiction.

Inequality \eqref{H_1_s} still holds if we replace $x_i$ by an arbitrary $x\in M$. Then, $C'$ does not depend on $x$. Moreover, choose $s$ large enough we then have for any $\gamma\in (0,1)$ that $\Vert \phi\Vert_{C^{0,\gamma}(B_{\rho/2}(x))}<\infty$ for all $x\in M$ and $\lim_{x\to \infty} \Vert \phi\Vert_{C^{0,\gamma}(B_{\rho/2}(x))}=0$. Thus, $\phi\in C^{0,\gamma}$ for any $\gamma\in (0,1)$. 
Then, by a further bootstrap step we obtain the same statement for $C^{1,\gamma}$ instead of $C^{0,\gamma}$ and for $\rho/3$ instead of $\rho/2$. Thus, for sufficiently large compactum $\hat{K}$ the norm $\Vert \phi\Vert_{C^{1,\gamma}(M\setminus\hat{K})}$ is arbitrarily close to zero. This implies the lemma.
\end{proof}

\begin{corollary}\label{improved-reg}  Let $(M^m,g)$ be of bounded geometry.
 Let $\phi\in L^q\cap C^0$  be a weak solution of $D\phi=\lambda |\phi|^{\frac{2}{m-1}}\phi$ with $\Vert \phi\Vert_p=1$. Then, $\phi\in C^{2,\gamma}$ for all $\gamma\in (0, \frac{2}{m-1}]$ if $m\geq 4$ and all $\gamma\in (0, 1)$ otherwise.
 \end{corollary}

\begin{proof}Let $\beta:=p-2=\frac{2}{m-1}$ and $\psi=|\phi|^\beta\phi$.  
At first we will show that 
\begin{equation}\label{napsi}
\nabla \psi= |\phi|^\beta \nabla \phi + \beta \< \nabla\phi,\phi\> |\phi|^{\beta-2}\phi
\end{equation} 
is in $C^\gamma$ for $\gamma$ as above: By Lemma~\ref{0-bdd-sol} $\phi\in C^{1,\alpha}$ for all $\alpha\in (0,1)$.
Thus, $\phi$ is locally Lipschitz and, hence, $|\phi|^\beta$ is in $C^\beta$. 
Moreover, $\nabla \phi\in C^\alpha$, thus the first summand in~\eqref{napsi}
is $C^{\min\{\alpha, \beta\}}$.
By \cite[Lem.~B.1]{Amm03a} 
$|\phi|^{\beta-2}\phi\otimes \phi\in C^\beta$.  
It follows that $\< \nabla\phi,\phi\> |\phi|^{\beta-2}\phi$ is $C^\ga$ as well.
Thus, $\nabla \psi\in C^\gamma$ and $\psi\in C^\alpha$ for all $\alpha\in (0,1)$.
Now Schauder estimates, see Theorem~\ref{est_iii}, imply $\phi\in C^{2,\gamma}$.
The corollary then follows.
\end{proof}

\begin{example}
Let us consider Euclidean $\mR^m$, $m\geq 2$ with standard basis $(e_i)_{i=1,\ldots, m}$ and 
with a parallel spinor $\psi_0\neq 0$.
We define 
  $$\phi(x^1,\ldots,x^m):= x^1 e_1\cdot \psi_0 - x^2 e_2\cdot \psi_0.$$
Then $\nabla \phi=dx^1\otimes e_1\cdot \psi_0 - dx^2 \otimes e_2\cdot \psi_0$,
and thus $D\phi= e_1\cdot e_1\cdot \psi_0 -  e_2\cdot e_2\cdot \psi_0 = -\psi_0+\psi_0=0$. Thus this spinor satisfies \eqref{eq_we} with $\la=0$, but is not 
$L^q$ and many conclusions in this section, in particular the $L^\infty$-bound, do not hold. 
The example thus shows that the $L^q$-condition in the above 
lemmata is necessary.
\end{example}

We know that by Lemma~\ref{bdd-sol} $\phi$ is in $L^\infty$. However, the following example shows that 
we cannot derive an upper bound for 
$\|\phi\|_{L^\infty}$ which only depends 
on $(M,g)$, $\|\phi\|_{L^q}$ and $\lambda$.  

\begin{example}\label{no_Linfty_bound}
Consider again Euclidean $\mR^m$. Take a Killing spinor $\phi$ on the sphere $(\mS^m, \sigma^m)$ normalized such that its $L^{q=\frac{2m}{m-1}}$-norm is one. Then on $\mS^m$ we have $D^{\mS^m}\phi=\frac{m}{2}\phi=\lm(\mS^m)|\phi|^{q-2}\phi$, cf. Remark~\ref{inv_sphere}. The stereographic projection ${\rm h}$ is a conformal map from the sphere with a point removed to the Euclidean space. Let now ${\rm h}^\rho$ be the composition of $\phi$ and the scaling of $\mR^m$ by $\rho$. Then, $g_E = f_\rho^2\, {\rm h}^\rho_*( \sigma^m)$ where $f_\rho=\rho^\frac{1}{2} f_1$ is the conformal factor. Using the identification of spinor bundles of conformal metrics, cf. \cite[Sect.~4]{hijazi}, we get a spinor $\tilde{\phi}=f_\rho^{-\frac{m-1}{2}} \phi$ fulfilling $D^{\mR^m}\tilde{\phi}= \lm(\mS^m)|\tilde{\phi}|^{q-1}\tilde{\phi}$ and $\Vert \tilde{\phi}\Vert_{L^q(\mR^m)}=1$. But $\Vert \tilde{\phi}\Vert_{L^\infty(\mR^m)}=\rho^{-\frac{m-1}{4}} \Vert f_1{-\frac{m-1}{2}} \phi\Vert_{L^\infty(\mR^m)}\
\to \infty$ as $\rho\to 0$.
We obtain an example where $L^\infty$-norm of solutions cannot be controlled in terms of its $L^q$-norm, $\lambda$ and $(M,g)$.
\end{example}

We close this section by some lemmata on removal of singularities for our Euler-Lagrange equations.

\begin{lemma}\label{removal-D}
 Let $(M,g)$ be an $m$-dimensional Riemannian spin manifold, and let $S\subset M$ be an embedded submanifold of dimension $\ell\leq m-s^*$ where $s^*$ is the conjugate exponent of $s$. Assume that $\phi$ is a spinor field such that $\Vert \phi\Vert_s<\infty$ for $s\in (1,\infty)$ and $D\phi=\lambda |\phi|^{s-2}\phi$ weakly on $M\setminus S$ for $\lambda\in \mR$. Then $D\phi=\lambda |\phi|^{s-2}\phi$ weakly on $M$.
\end{lemma}

\begin{proof}
 We follow the proof for the removal of singularities for weakly harmonic spinors in \cite[Lem.~2.4]{ammann.dahl.humbert:09}: Let $U_S(\epsilon)$ consist of all points of $M$ with distance $\leq \epsilon$ to $S$. Let $\eta_\delta$ be a cut-off function with $\eta_\delta=1$ on $U_S(\delta)$, $\eta_\delta=0$ on $U_S(2\delta)$ and $|\nabla \eta_\delta|\leq 2\delta^{-1}$. Then, we obtain for a smooth and compactly supported spinor $\psi$ on $M$
 \begin{align*}
   \int_M \left< \phi, D\psi \right>-\lambda  \int_M \left< |\phi|^{s-2}\phi, \psi \right> &= \int_M \left< \phi, D (1-\eta_\delta) \psi \right>-\lambda  \int_M \left< |\phi|^{s-2}\phi, (1- \eta_\delta)\psi \right>\\
   & + \int_M \left< \phi, \eta_\delta D\psi \right>+\int_M \left< \phi, \nabla \eta_\delta \cdot \psi \right>-\lambda  \int_M \left< |\phi|^{s-2}\phi, \eta_\delta \psi \right>.
 \end{align*}
 The sum of the first two summands on the right side vanishes since $D\phi=\lambda |\phi|^{s-2}\phi$ weakly on $M\setminus S$. 
Moreover, $\left| \int_M \left< \phi, \eta_\delta D\psi \right>\right|\leq \Vert \phi\Vert_s \Vert D\psi\Vert_{L^{s^*}(U_S(2\delta))} \to 0$ and  $\left| \int_M \left< |\phi|^{s-2}\phi, \eta_\delta \psi \right>\right| \leq \Vert \phi\Vert_{s}^{s/s^*} \Vert \psi\Vert_{L^{s}(U_S(2\delta))}\to 0$ as $\delta\to 0$.  The remaining term can be estimates by

\begin{align*}
 \left| \int_M \left< \phi, \nabla \eta_\delta \cdot \psi \right>\right|&\leq \frac{2}{\delta} \Vert \phi\Vert_{L^{s}(U_S(2\delta))} \Vert \psi\Vert_{L^{s^*}(U_S(2\delta))}\leq \frac{C}{\delta}\Vert\phi\Vert_{L^{s}(U_S(2\delta))}  {\rm vol}( U_S(2\delta)\cap \supp \psi)^{\frac{1}{s^*}}\\
 &\leq C'\underbrace{\Vert\phi\Vert_{L^{s}(U_S(2\delta))}}_{\to 0} \delta^{\frac{m-\ell}{s^*}-1} \to 0. 
\end{align*}

\end{proof}

\begin{lemma}\label{removal-L} 
 Let $(M,g)$ be an $m$-dimensional Riemannian spin manifold, and let $S\subset M$ be an embedded submanifold of dimension $\ell\leq m-2s^*$ where $s^*$ is the conjugate exponent of $s$. Assume that $v$ is a nonnegative function such that $\Vert v\Vert_s<\infty$ for $s\in (1,\infty)$ and $Lv=\mu v^{s-1}$ weakly on $M\setminus S$ for $\mu\in \mR$. Then $Lv=\mu v^{s-1}$ weakly on $M$.
\end{lemma}

\begin{proof}
The proof is similar to the one of  Lemma~\ref{removal-D}, and we use the notations therein. The cut-off function $\eta_\delta$ is chosen such it fulfills additionally $|\Delta \eta_\delta|\leq 4\delta^{-2}$. Then, the estimates are done analogously.
\end{proof}

\section{Gromov-Hausdorff convergences}
Let $(M_i,g_i,x_i)$, $i\in \mathbb{N}$, and $(M_\infty,g_\infty,x_\infty)$
be pointed complete connected Riemannian manifolds.
We say that $(M_i,g_i,x_i)$ converges to $(M_\infty,g_\infty,x_\infty)$ in the
$C^k$-topology of pointed Riemannian manifolds if for every $R>0$ and every 
$i\geq i_0(R)$ there is an injective immersion 
$\phi_i^R\colon B^{M_\infty,g_\infty}_{R+1}(x_\infty)\to B^{M_i,g_i}_{R+1}(x_i)$ such that 
$(\phi_i^R)^*g_i$ converges to $g_\infty$ on
$B^{M_\infty,g_\infty}_R(x_\infty)$ in the $C^k$-topology.
If all manifolds above carry spin structures, then we say 
that they converge in the $C^k$-topology of pointed Riemannian spin 
manifolds if additionally the maps $\phi_i^R$ preserve the chosen spin 
structures.

\begin{lemma}\label{lemma.pointed.conv}
If $(M_i,g_i,x_i)$ converges to $(M_\infty,g_\infty,x_\infty)$ in the
$C^2$-topology of pointed Riemannian manifolds, then 
 $$\limsup_{i\to \infty} \Q(M_i,g_i) \leq \Q(M_\infty,g_\infty).$$
If $(M_i,g_i,x_i)$ converges to $(M_\infty,g_\infty,x_\infty)$ in the
$C^1$-topology of pointed Riemannian spin manifolds, then 
 $$\limsup_{i\to \infty} \Qs(M_i,g_i) \leq \Qs(M_\infty,g_\infty).$$
\end{lemma}

\begin{proof}
For a given $\ep>0$ we take $v\in C^\infty_c(M_\infty)$ with
$\cF^{g_\infty}(v)< \Q(M_\infty,g_\infty)+\ep$. Choose $R>0$ such that
the support of $v$ is contained in  $B^{M_\infty,g_\infty}_R(x_\infty)$. 
For sufficiently large $i$ we then have 
\[ \Q(M_i,g_i)\leq \cF^{g_i}(v\circ(\phi_i^R)^{-1})=\cF^{(\phi_i^R)^*g_i}(v)\leq 
  \cF^{g_\infty}(v)+\ep < \Q(M_\infty,g_\infty)+2\ep\]
  where the second inequality uses that $\mathcal{F}^g$ depends only on derivatives of $g$ up to order $2$. The first part of the lemma follows in the limit $\ep\to 0$.

The spinorial statement is proven completely analogously. Here, convergence in $C^1$ is enough since the Dirac operator is of first order.
\end{proof}

In the articles
\cite{ADH} 
 and 
\cite{ADH09} 
the following situation was considered. Assume that $N^m$ is obtained
from $M^m$ by a surgery of dimension $k$. Then for any metric $g$
on $M$ a family of special metrics $g_\th$, $\theta>0$, was constructed.
It was proved in \cite{ADH} in combination with estimates given in
\cite{ammann.dahl.humbert:p11b} that for all $k\leq m-4$ and all 
$k=m-3\leq 3$ we have
  $$\lim_{\theta\to 0} \Q(N,g_\th)\geq \min\{\Q(M,g),\wLa_{m,k}\}.$$
Similarly it was proven in \cite{ADH09} for $k\leq m-2$ that 
  $$\lim_{\theta\to 0} \Qs(N,g_\th)\geq \min\{\Qs(M,g),\wLas_{m,k}\}.$$

We apply this construction to $M=\mS^m$ equipped with the standard metric $g=\sigma^m$.
Then $N=S^{k+1}\times S^{m-k-1}$. Thus we obtain a family of metrics $g_\th$ on 
$N=S^{k+1}\times S^{m-k-1}$ with 
   $$\lim_{\theta\to 0} \Q(S^{k+1}\times S^{m-k-1},g_\th)\geq \wLa_{m,k}\qquad \mbox{if }k\leq m-4\mbox{ or if }k=m-3\leq 3,$$
and 
   $$\lim_{\theta\to 0} \Qs(S^{k+1}\times S^{m-k-1},g_\th)\geq \wLas_{m,k}\qquad \mbox{if }k\leq m-2.$$

The following lemma is proven with exactly the same methods as in Subsection~6.3 of \cite{ADH}.
\begin{lemma}\label{lemma.all.c}
For any $c\in [0,1]$, there are points $x_\th\in S^{k+1}\times S^{m-k-1}$, $ \th\in (0,1)$, such that $(S^{k+1}\times S^{m-k-1},g_\th,x_\th)$ converges in the $C^\infty$-topology of pointed Riemannian manifolds
to $(\Mc^{m,k},x_0)$ where $x_0$ is an arbitrary base point.
\end{lemma}
Lemmata~\ref{lemma.pointed.conv} and~\ref{lemma.all.c} imply
 \[\lim_{\theta\to 0} \Q(S^{k+1}\times S^{m-k-1},g_\th)\leq  \Q(\Mc^{m,k}) \]
and 
 \[\lim_{\theta\to 0} \Qs(S^{k+1}\times S^{m-k-1},g_\th)\leq \Qs(\Mc^{m,k}) \]
for all $c\in[0,1]$ with the same restrictions on $k$ as above.  Hence, we immediately obtain 
\begin{proposition} \label{wLas_leq_Las}
  \[\wLa_{m,k}\leq  \sLa_{m,k}\qquad \mbox{ for } k\leq m-4 \mbox{ and for } k=m-3\leq 3,\]
  \[\wLas_{m,k}\leq \sLas_{m,k} \qquad \mbox{ for } k\leq m-2.\]
\end{proposition}

Note that in this proposition we do not get any statement about the invariants for $\Mc^{m,k}$ for a fixed $c$, compare to Corollary~\ref{wQs_leq_Qs}.

\section{Cut-off arguments}\label{sec_cut}

In this section we use cut-off functions to compare the $*$-invariants (which are defined as the infimum of a functional) with there $\sim$-counterparts (which are defined as the infimum of nonlinear eigenvalues).  

\begin{lemma}
Let $(M^m,g)$ be a complete connected $m$-dimensional Riemannian manifold. Then, $\Q (M,g)\leq \wQ (M,g)$.
\end{lemma}

\begin{proof} (cp. \cite[Lem.~3.5]{ADH}) Let $v\in C^2(M)\cap L^\infty(M)\cap L^2(M)$, $v\geq 0$, satisfying $L^gv=\mu_v v^{p-1}$ with $\mu_v\in \R_{\geq 0}$ and $\Vert v\Vert_{L^p}=1$ where $p=\frac{2m}{m-2}$. 
 We fix $z\in M$. Let $\eta_r$ be a smooth cut-off function with values in $[0,1]$, $\eta_r=0$ on $M\setminus B_{2r}(z)$, $\eta_r=1$ on $B_{r}(z)$, and $|d\eta_r|\leq 2r^{-1}$. Then, 
 
 \begin{align*}
  \Q(M,g)\leq & \frac{\int_M \eta_r v L^g(\eta_r v) \vo_g}{\Vert \eta_r v\Vert_{L^p(g)}^2}= \frac{\int_M \eta_r^2vL^gv+a|d\eta_r|^2v^2 \vo_g}{\Vert \eta_r v\Vert_{L^p(g)}^2}\\
  \leq & \frac{\int_M \mu_v \eta_r^2v^p+a4r^{-2}v^2 \vo_g}{\Vert \eta_r v\Vert_{L^p(g)}^2} \to \mu_v \text{\ as\ } r\to \infty.
  \end{align*}

\end{proof}

\begin{corollary}
\label{La_leq_wLa}
$\Q(\Mc^{m,k})\leq \wQ(\Mc^{m,k})$ and $\sLa_{m,k}\leq \wLa_{m,k}$ for all $m,k$.
\end{corollary}

\begin{lemma}\label{lem_comparelm1}
Let $(M^m,g)$ be a complete connected $m$-dimensional Riemannian spin manifold. Assume that $D$ is $L^{q^*=\frac{2m}{m+1}}$-invertible. Then $\Qs (M,g)\leq \wQs (M,g)$.
\end{lemma}

\begin{proof}
Let $\lambda=\wlm (M,g)$. By definition of $\wlm$, cf. \eqref{def_solLa}, there is a $\phi\in L^2\cap L^\infty\cap C^1$ with $D\phi=\lambda |\phi|^{q-2}\phi$ and $\Vert \phi\Vert_q=1$ where $q=\frac{2m}{m-1}$. Then, $D\phi\in L^{q^*}$, and by the $L^{q^*}$-invertibility of $D$ we get that $\phi\in L^{q^*}$. Hence, $\lambda >0$ since otherwise $\phi$ would be a nonzero $L^{q^*}$-harmonic spinor which contradicts the $L^{q^*}$-invertibility.

We fix $z\in M$. Let $\eta_r$ be a smooth cut-off function with values in $[0,1]$, $\eta_r=0$ on $M\setminus B_{2r}(z)$, $\eta_r=1$ on $B_{r}(z)$, and $|\d\eta_r|\leq 2r^{-1}$. Then 
\[\underbrace{(D(\eta_r \phi), \eta_r\phi)}_{\in \R}= (\d \eta_r\cdot 
\phi, \phi)+\underbrace{(\eta_r^2D\phi,\phi)}_{\in \R}=\lambda \int_M \eta_r^2 |\phi|^q\vo_g>0\] where we used that the summand including $\d\eta_r$ vanishes due to $\< \d \eta_r\cdot \phi, \phi\>_x\in \rm{i}\R$. Thus,

\begin{align*} 
 \lm (M,g,\chi)&\leq \frac{\Vert D(\eta_r \phi)\Vert_{q^*}^2}{(D(\eta_r \phi),\eta_r \phi) }\leq \frac{\left( \Vert \d \eta_r\cdot \phi\Vert_{q^*} + \Vert\eta_r D \phi\Vert_{{q^*}}\right)^2}{\lambda \int_M \eta_r^2 |\phi|^{q}\vo_g }\\
&\leq \frac{\left( \frac{2}{r} \Vert \phi\Vert_{q^*} + \lambda \left(\int_M \eta_r^{q^*} |\phi|^{q}\vo_g\right)^{1/q^*}\right)^2}{\lambda \int_M \eta_r^2 |\phi|^q\vo_g }\\
&\to \lambda\, \Vert \phi\Vert_{q}^{q(2-q^*)/q^*}=\lambda\,\Vert \phi\Vert_{q}^{\frac{2}{m-1}}\leq \lambda
\end{align*}
as $r\to \infty$. Note that the summand $\frac{1}{r}\Vert \phi\Vert_{q^*}\to 0$ since $\phi\in L^{q^*}$ as shown above. Hence, $\Qs\leq \wQs$.
\end{proof}

\begin{corollary}\label{Qs_leq_wQs} For all $c\in [0,1]$ and $k\leq m-1$, we have $\Qs(\Mc^{m,k}, g_c)\leq \wQs(\Mc^{m,k}, g_c).$ In particular, $\sLas_{m,k}\leq \wLas_{m,k}$.
\end{corollary} 

\begin{proof} We start with $k\leq m-2$. Lemma~\ref{lem_comparelm1} and Proposition~\ref{inv_Mc} imply $\Qs(\Mc^{m,k}, g_c)\leq \wQs(\Mc^{m,k}, g_c)$ for all $\frac{m-k-1}{2}> ck (\frac{m+1}{2m}-\frac{1}{2})=\frac{ck}{2m}$, i.e., for all $k \leq m-2$ and $c\in [0,1]$.

The remaining case $k=m-1$ follows directly from Lemma~\ref{1_wQs}.
\end{proof}

\section{\texorpdfstring{The model space $\mathbb{M}_1^{m,k}$}{The model spaces for c=1}}\label{c=1}

For $c=1$ the model spaces $\Mc^{m,k}$ is very special: The manifold \[(\mathbb{M}_1^{m,k}=\mH^{k+1}\times \mS^{m-k-1}, g_1=g_{\mH^{k+1}}+\sigma^{m-k-1}= \sinh^2 t\, \sigma^{k+1} +dt^2+\sigma^{m-k-1})\] is conformal to $(\mS^m\setminus \mS^k, \sigma^m)$, \cite[Prop.~3.1]{ADH},
 \begin{equation}\label{conf_M1} \mathfrak{u}: \mH^{k+1}\times \mS^{m-k-1} \to \mS^m\setminus \mS^k,\  g_1=f^2\mathfrak{u}^* \sigma^m\ {\rm where\ } f=f(t)=\cosh^2 t\end{equation}
 where $\cosh t=(\sin r)^{-1}$ with $r=\dist(. , \mS^k)$.

Using this conformal map, we will immediately obtain some of the $Q$-invariants of $\mathbb{M}_1^{m,k}$.

\begin{lemma}\label{1_QS}
 $\Q(\mathbb{M}_1^{m,k})=\Q(\mS^m)=\Qs (\mathbb{M}_1^{m,k})$.
\end{lemma}

\begin{proof} By conformal invariance  $\Q(\mathbb{M}_1^{m,k}, g_1)=\Q(\mS^m\setminus \mS^k, \sigma^m)$. Since $\Q$ is defined over test functions, we have $\Q(\mS^m\setminus \mS^k)\geq \Q(\mS^m)$. On the other hand $\Q(\mS^m)$ is the highest possible value for $\Q$, see Remark~\ref{inv_sphere}, and thus $\Q(\mathbb{M}_1^{m,k})=\Q(\mS^m)$. With analogous arguments one gets $\Qs(\mathbb{M}_1^{m,k})=\Qs(\mS^m)$. Together with $\Q(\mS^m)=\Qs(\mS^m)$ the lemma follows.
\end{proof}

In order to examine $\wQ(\mathbb{M}_1^{m,k})$ and $\wQs (\mathbb{M}_1^{m,k})$ we will need modifications of the removal of singularities results in Lemma~\ref{removal-D} and~\ref{removal-L}.

 \begin{lemma}\label{removal-D+}
 Let $(M,g)$ be an $m$-dimensional Riemannian spin manifold, and let $S\subset M$ be an embedded submanifold of dimension $\ell\leq m-1$. Assume that $\phi$ is a spinor field such that $\int_{U_\epsilon(S)} \frac{1}{\rho} |\phi|^2<\infty$ where $U_\epsilon(S)$ consists of all points of $M$ with distance $\rho\leq \epsilon$ to $S$.  Moreover, let $D\phi=\lambda |\phi|^{q-2}\phi$ weakly on $M\setminus S$ for $\lambda>0$. Then $D\phi=\lambda |\phi|^{q-2}\phi$ weakly on~$M$.
\end{lemma}

\begin{proof}
 We adapt the proof of Lemma~\ref{removal-D}: Let $\tilde{\eta}_\delta$ be the function on $M$ defined by
 
  \[ \tilde{\eta}_{\delta}(x)= \left\{  \begin{matrix} 
                              0&  \text{for\ }\rho:=\dist(x,S) \geq \rho_0:=\delta\\
                              \delta \log (\rho_0/\rho) & \text{for\ } \rho_0\geq \rho\geq \rho_1:= \rho_0 e^{-1/\delta}\\
                              1 &  \text{for\ } \rho_1\geq \rho.
                             \end{matrix}
               \right.\]
 We smooth out $\tilde{\eta}_\delta$ in such a way that the resulting function $\eta_\delta$ still fulfills $\eta_\delta(x)=1$ for $\rho\geq \rho_0$, $\eta_\delta(x)=0$ for $\rho\leq \rho_1$, and $|\nabla\eta_\delta|\leq \frac{2\delta}{\rho}$.
 
 Then, for a smooth and compactly supported spinor $\psi$ on $M$ we obtain

  \begin{align*}
    \int_M \left< \phi, D\psi \right>-\lambda  \int_M \left< |\phi|^{q-2}\phi, \psi \right> &= \int_M \left< \phi, D (1-\eta_\delta) \psi \right>-\lambda  \int_M \left< |\phi|^{q-2}\phi, (1- \eta_\delta)\psi \right>\\
    & + \int_M \left< \phi, \eta_\delta D\psi \right>+\int_M \left< \phi, \nabla \eta_\delta \cdot \psi \right>-\lambda  \int_M \left< |\phi|^{q-2}\phi, \eta_\delta \psi \right>.
  \end{align*}
  The sum of the first two summands on the right side vanish because the equation holds on $M\setminus S$. The terms $\int_M \left< \phi, \eta_\delta D\psi \right>$ and $ \int_M \left< |\phi|^{q-2}\phi, \eta_\delta \psi \right>$ vanish for the same reason as in the proof of Lemma~\ref{removal-D}. The remaining term is now estimated by
 \begin{align*}
  \left| \int_M \left< \phi, \nabla \eta_\delta \cdot \psi \right>\right|&\leq C\delta  \int_{U_\delta(S)\cap \supp \psi} \frac{1}{\rho}|\phi|  \leq C\delta \left(\int_{U_\delta(S)} \frac{1}{\rho}|\phi|^2\right)^{\frac{1}{2}}   \left(\int_{U_\delta(S)\cap \supp \psi} \frac{1}{\rho}\right)^{\frac{1}{2}}\\
 &\leq C'\delta \underbrace{\left(\int_{U_\delta(S)} \frac{1}{\rho}|\phi|^2\right)^{\frac{1}{2}}}_{\to 0\text{\ as\ }\delta\to 0}   \underbrace{\left(\int_{\delta e^{-1/\delta}}^\delta  \frac{1}{\rho}\rho^{m-\ell-1}\, \d\rho\right)^{\frac{1}{2}}}_{
 \text{is\ } \delta^{-1/2} \text{\ for $m=\ell+1$ and\ }\leq \hat{C} \delta^{(m-\ell-1)/2} \text{\ else}}\to 0 
 \end{align*}
as $\delta\to 0$ which concludes the proof.
\end{proof}

\begin{lemma}\label{removal-L+}
 Let $(M,g)$ be an $m$-dimensional Riemannian spin manifold, and let $S\subset M$ be an embedded submanifold of dimension $\ell\leq m-2$. Assume that $v$ be a nonnegative function such that $\int_{U_\epsilon(S)} \frac{1}{\rho^2} v^2<\infty$ where $U_\epsilon(S)$ consists of all points of $M$ with distance $\rho\leq \epsilon$ to~$S$.  Moreover, let $Lv=\mu v^{p-1}$ weakly on $M\setminus S$. Then $Lv=\mu v^{p-1}$ weakly on $M$.
\end{lemma}

\begin{proof}We use an analogous argumentation as in the proof above. Now, we smooth out $\tilde{\eta}_\delta$ in such a way that the resulting $\eta_\delta$ fulfills additionally $|\Delta \eta_\delta|\leq \frac{4\delta}{\rho^2}$. Then, for $h\in C_c^\infty(M)$ we estimate $ \int_M v Lh- \int_M v^{p-1}h$ in a similar way -- only $\Delta \eta_\delta$ gives rise to a new term:

 \begin{align*}
  \left| \int_M v h \Delta \eta_\delta \right|&\leq C\delta  \int_{U_\delta(S)\setminus U_{\delta e^{-1/\delta}}(S)\cap \supp h} \frac{1}{\rho^2}v \\
  &\leq C\delta \left(\int_{U_\delta(S)} \frac{1}{\rho^2}v^2\right)^{\frac{1}{2}}   \left(\int_{U_\delta(S)\setminus U_{\delta e^{-1/\delta}}(S)\cap \supp h} \frac{1}{\rho^2}\right)^{\frac{1}{2}}\\
 &\leq C'\delta \underbrace{\left(\int_{U_\delta(S)} \frac{1}{\rho^2}v^2\right)^{\frac{1}{2}}}_{\to 0\text{\ as\ }\delta\to 0}   \underbrace{ \left(\int_{\delta e^{-1/\delta}}^\delta  \frac{1}{\rho^2}\rho^{m-\ell-1}\, d\rho\right)^{\frac{1}{2}} }_{ \text{is\ } \delta^{-1/2} \text{\ for $m=\ell+2$ and\ }\leq \hat{C} \delta^{(m-\ell-1)/2} \text{\ else}}  \to 0 \text{\ as\ }\delta\to 0.
 \end{align*}
 
\end{proof}

\begin{lemma}\label{1_wQs}For $m\geq 2$
\begin{align*} \wQs(\mathbb{M}_1^{m,k})=&\left\{ \begin{matrix}\Q(\mS^m) & \text{for\ } k\leq m-2\\
 \wQs(\mH^m)=\infty & \text{for\ } k= m-1.
 \end{matrix} \right.
\end{align*}
 \end{lemma}

\begin{proof}
Let $\phi\in L^\infty\cap L^2\cap C^1$ be a solution of $D\phi=\lambda|\phi|^{q-2}\phi$ on $\mathbb{M}_1^{m,k}$ with $0<\Vert \phi\Vert_{L^q}\leq 1$. Using the conformal map $\mathfrak{u}$ in \eqref{conf_M1} we obtain a $C^1$-solution $\tilde{\phi}=f^{\frac{m-1}{2}}\phi$ of $D^{\sigma^m}\tilde{\phi}=\lambda|\tilde{\phi}|^{q-2}\tilde{\phi}$ on $\mS^m\setminus \mS^k$ with $0<\Vert \tilde{\phi}\Vert_{L^q}\leq 1$. Moreover, since ${\phi}$ is $L^2$ we get $\infty> \Vert \phi\Vert_{L^2}^2= \int_{\mS^m\setminus \mS^k} f^{-1}|\tilde{\phi}|^2\, \vo_{\sigma^m} = \int_{\mS^m\setminus \mS^k} (\sin \rho)^{-1}|\tilde{\phi}|^2\, \vo_{\sigma^m}$. In particular, $\tilde{\phi}\in L^2$. Moreover, $\frac{1}{\rho}-\frac{1}{\sin \rho}$ is bounded  by $\mathcal{O}(\epsilon)$ for $\rho \in (0,\epsilon)$. Thus, $\int_{U_\epsilon (\mS^k)} \frac{1}{\rho}|\tilde{\phi}|^2\, \vo_{\sigma^m}<\infty$ as well. Because of Lemma~\ref{removal-D+} $\tilde{\phi}$ solves $D^{\sigma^m}\tilde{\phi}=\lambda|\tilde{\phi}|^{q-2}\tilde{\phi}$ weakly on 
all of $\mS^m$. By regularity theory on compact manifolds  $\tilde{\phi}\in L^q$ implies $\tilde{\phi}\in H_1^q\subset  H_1^{q^*}$. Hence, $\tilde{\phi}$ can serve as a test function for $\mathcal{F}^{\rm spin}$ on $\mS^m$ which implies $\lambda\geq \wQs(\mS^m)=\Q(\mS^m)$. Thus, $\wQs( \mathbb{M}_1^{m,k})\geq \Q(\mS^m)$.

 Let now $\tilde{\phi}$ be a Killing spinor on $\mS^m$ with Killing constant $-\frac{1}{2}$ and $\Vert \tilde{\phi}\Vert_{L^q(\mS^m)}=1$. Then $D\tilde{\phi}= \Qs (\mS^m) |\tilde{\phi}|^{\frac{4}{m-1}}\tilde{\phi}$. Then using the identification of spinor bundles to conformal metrics as in Example \ref{no_Linfty_bound} the spinor ${\phi}=f^{-\frac{m-1}{2}} \tilde{\phi}$ fulfills the Euler-Lagrange equation for $D$ on $\mathbb{M}_1^{m,k}$ and is in $L^\infty\cap L^{q}$. Moreover, if $m-k\geq 2$, then 
 \[\Vert {\phi}\Vert_{L^2}^2=C_2 \int_0^\infty \cosh^{1-m} t \sinh^k t\, \d t \leq C_3 +C_4 \int_1^\infty e^{(1-m+k) t}\, \d t <\infty.\] 
 
 Thus, for $k\leq m-2$ we obtained $\wQs (\mathbb{M}_1^{m,k})=\Q(\mS^m)$.  
 
 Let now $k=m-1$. Then, $\mathbb{M}_1^{m,m-1}$ corresponds to two copies of the hyperbolic space. Thus, $\wQs (\mathbb{M}_1^{m,m-1})=\wQs(\mH^m)$. Let $\phi$ now be a solution as above on $\mH^m$. By a conformal map we get as above a solution $\tilde{\phi}$ on the lower hemisphere of $\mS^m$. Extending $\tilde{\phi}$ by zero to all of $\mS^m$, we obtain a weak solution to our nonlinear Dirac eigenvalue equation on $\mS^m\setminus \mS^{m-1}$. Using again Lemma~\ref{removal-D+} we see that $\tilde{\phi}$ is already a nontrivial weak solution on all of $\mS^m$. But since $\tilde{\phi}$ vanishes on an open subset this contradicts the unique continuation principle, \cite{boosbavnbek_marcolli_wang_02}. Thus, such a solution $\phi$ we started with cannot exist. Hence, $\wQs(\mH^m)=\wQs (\mathbb{M}_1^{m,m-1})=\infty$.
 \end{proof}
 
 \begin{lemma}\label{1_wQ}For $m\geq 3$
\[  \wQ(\mathbb{M}_1^{m,k})=\left\{ \begin{matrix}\Q(\mS^m) & \text{for\ } k\leq m-3\\
 \infty & \text{for\ } k= m-2\\
\wQ(\mH^m)=\infty & \text{for\ } k= m-1.\\
  \end{matrix} \right. \]
 \end{lemma}

 \begin{proof}We start analogously as in the spin case from above with a nonnegative solution $v\in L^\infty\cap L^2\cap C^2$ of $Lv=\mu v^{p-1}$ on $\mathbb{M}_1^{m,k}$ and use the conformal map $\mathfrak{u}$ in \eqref{conf_M1} to obtain $\tilde{v}$ on $\mS^m\setminus \mS^k$. Analogous as in the proof of Lemma~\ref{1_wQs} we see that $\int_{U_\epsilon(\mS^k)} \frac{1}{\rho^2} \tilde{v}^2\, \vo_{\sigma^m}<\infty$ which allows to use Lemma~\ref{removal-L+} for $k\leq m-2$. Thus,  we get as in the last lemma that $\wQ(\mathbb{M}_1^{m,k})\geq \Q(\mS^m)$ for $k\leq m-2$.  On the other hand,  $\tilde{v}=\rm const$ such that $\Vert \tilde{v}\Vert_{L^p(\mS^m)}=1$ is a solution of the Euler-Lagrange equation on $\mS^m$. Set ${v}= f^{-\frac{m-2}{2}}\mathfrak{u}^*\tilde{v}= C \cosh^{-\frac{m-2}{2}} t$ where $C$ is an appropriate constant. Then by conformal invariance, ${v}$ fulfills the Euler-Lagrange equation on $\mathbb{M}_1^{m,k}$ and is in $L^p$. Moreover, if $m-k\geq 3$, ${v}\in L^2$ as can be seen by $\Vert {v}\Vert_{
L^2}^2=C \int_0^\infty \cosh^{2-m} t \sinh^k t\, \d t \leq C_1+C_2 \int_1^\infty e^{(2-m+k) t}\, \d t <\infty$. Hence, for $k\leq m-3$ we have $\wQ(\mathbb{M}_1^{m,k})=\Q(\mS^m)$. 
 
For $m-k\leq 2$ we obtained up to now that each  nonnegative solution $v$ on $\mathbb{M}_1^{m,k}$ gives rise to a nonnegative solution $\tilde{v}$ on $\mS^m$. By \cite[Thm.~5]{farid} $\tilde{v}$ is continuous and everywhere positive.  
For $m-k=2$  and using that $\tilde{v}$ is continuous and positive we can estimate 
\[\int_{U_\epsilon(\mS^{m-2})} \frac{1}{\rho^2} \tilde{v}^2 \vo_{\sigma^m} \geq  C \int_0^\epsilon \frac{1}{\rho^2}\rho \, \d\rho.\] Thus, the left integral is not finite which gives a contradiction. Thus, $\wQ(\mathbb{M}_1^{m,m-2})=\infty$.

For $k=m-1$ let $v\in L^\infty\cap L^2\cap C^2$ be a positive solution of $Lv=\mu v^{p-1}$ on $\mathbb{M}_1^{m,m-1}$. Thus, we have two solutions of the same equation on the hyperbolic space. We will show that a nontrivial solution of $Lv=\mu |v|^{p-2}v$ on the hyperbolic space cannot exist in $L^2$. From a solution on the hyperbolic space we can use the conformal map $\mathfrak{u}$ to obtain a solution $\tilde{v}$ on the lower hemisphere $\mS^m$. We extend $\tilde{v}$ to the upper hemisphere by reflection and changing its sign on the upper hemisphere. Thus, $\tilde{v}$ solves $L\tilde{v}=\mu |\tilde{v}|^{p-2}\tilde{v}$ on $\mS^m\setminus \mS^{m-1}$. Next we show that $\tilde{v}$ solves this equation weakly on all of $\mS^m$. Since $\tilde{v}$ is an odd function with respect to reflection at the equator, it suffices to test with odd functions $h\in C^\infty(\mS^m)$. Thus, there is a constant $C>0$  such that $|h(x)|\leq C \dist(x, \mS^{m-1})=C\rho$. Following the arguments in Lemma~\ref{removal-L+} the 
estimates are done analogously, and it remains to estimate

\begin{align*}
  \left| \int_M v h \Delta \eta_\delta \right|&\leq C\delta  \int_{U_\delta(\mS^{m-1})\setminus U_{\delta e^{-1/\delta}}(\mS^{m-1})\cap \supp h} \frac{1}{\rho^2}vh \leq C'\delta  \int_{U_\delta(\mS^{m-1})\setminus U_{\delta e^{-1/\delta}}(\mS^{m-1})\cap \supp h} \frac{1}{\rho}\\
 &\leq C''\delta \underbrace{ \left(\int_{\delta e^{-1/\delta}}^\delta  \frac{1}{\rho}\, d\rho\right)^{\frac{1}{2}} }_{ =\delta^{-1/2} }  \to 0 \text{\ as\ }\delta\to 0.
 \end{align*}
 
Thus, $\tilde{v}$ solves $L\tilde{v}=\mu |\tilde{v}|^{p-2}\tilde{v}$ weakly on $\mS^m$. Then, regularity theory implies that $\tilde{v}\in C^2$ and thus $\tilde{v}|_{\mS^{m-1}}=0$. Using a conformal transformation from the lower hemisphere to the disk~$D$ in $\mR^{m}$, we obtain a solution $\hat{v}$ of $L\hat{v}=\mu |\hat{v}|^{p-2}\tilde{v}$ on $D$ which is somewhere nonzero in the interior of $D$ and zero on the boundary. This is a contradiction to \cite{Po65}, \cite[Thm.~III.1.3]{struwe}. Thus the solution we started with cannot exist, and hence $\wQ(\mH^m)=\infty$.
\end{proof}

\section{\texorpdfstring{The invariants for $k=m-1$}{The invariants for k=m-1}}\label{km1}

The constants 
 $\sLa_{m,m-1}$ and $\sLas_{m,m-1}$ are easy to determine.

\begin{lemma}
We have
   $\sLas_{m,m-1}=\Q(\mS^m)$ for all $m\geq 3$ and $ \sLa_{m,m-1}=\Q(\mS^m)
   $ for all  $m\geq 2$.
\end{lemma}

\begin{proof}
We show
   $$\Qs(\mM_c^{m,m-1})=\Q(\mM_c^{m,m-1})=\Q(\mS^m).$$
For $c\neq 0$, our model space $\mM_c^{m,m-1}$ is isometric to two copies of the rescaled hyperbolic space $\mH_c^{m}$, and for $c=0$ it is isometric to two copies of the Euclidean $\mR^m$. Thus,  $\Q(\mM_c^{m,m-1})=\Q(\mH_c^{m})=\Q(\mH^{m})= \Q(\mR^m)=\Q(\mS^m)$, cf. Remark~\ref{inv_sphere} and the first equality follows from \cite[Lem.~1.10]{Kob_87}. Using \cite[Lem.~2.0.5]{Grossediss} the analogous equations hold for $\Qs$ which finishes the proof.
\end{proof}

For the $\tilde{Q}$-invariants we have by scaling and Lemma~\ref{1_wQs} that  
$\wQs(\mH^m_c)=\wQs(\mH^m)=\infty$  for $c\in (0,1]$ and $m\geq 2$, and $\wQ(\mH^m_c)=\wQ(\mH^m)=\infty$
for $c\in (0,1]$ and $m\geq 3$. It remains to consider the Euclidean space.

\begin{lemma}
We have $\wQ(\R^m)=\infty$ for $m=3,4$, $\wQ(\R^m)=\Q(\mS^m)$ for all $m\geq 5$, $\wQs(\mR^m)= \Q(\mS^m)$ for all  $m\geq 3$ and $\wQs(\mR^2)\geq \Q(\mS^2)$.   
\end{lemma}

\begin{proof} We start examining $\wQs(\mR^m)$. Let $\phi\in L^2\cap L^\infty\cap C^1$ be a solution  on $\mR^m$ of $D\phi=\lambda |\phi|^{q-2}\phi$ with $0<\Vert \phi\Vert_{L^q}\leq 1$ for some $\lambda>0$. 
By stereographic projection, we have $\sigma^m=\frac{4}{(1+r^2)^2}g_E$ where $r$ is the radial function in $\mR^m$. Using the conformal invariance of the nonlinear Dirac eigenvalue equation above, we get for $\tilde{\phi}=\left(\frac{1+r^2}{2}\right)^{(m-1)/2}\phi$ that $D^{\sigma^m}\tilde{\phi}=\lambda|\tilde{\phi}|^{q-2}\tilde{\phi}$ and $0\leq \Vert \tilde{\phi}\Vert_{L^q}\leq 1$ on $\mS^m\setminus\{N\}$. Moreover, $\int_{\mS^m\setminus\{N\}} \frac{1+r^2}{2} |\tilde{\phi}|^2\,\vo_{\sigma^m}=\int_{\mR^m} |\phi|^2 \,\vo_{E}<\infty$  and $1+r^2= \frac{2}{\sin^2 \rho}$ where $\rho$ is the distance to the north pole $N$. In particular, it now follows similarly  to the proof of Lemma~\ref{1_wQs} that $\int_{U_\epsilon (N)} \frac{1}{\rho^2}|\tilde{\phi}|^2 \,\vo_{\sigma^m}$ is finite. In particular, $\int_{U_\epsilon (N)} \frac{1}{\rho}|\tilde{\phi}|^2 \,\vo_{\sigma^m}$ is finite as well. Thus, we can apply Lemma~\ref{removal-D+} and see the nonlinear Dirac eigenvalue equation from above is valid on all 
of $\mS^m$. Thus, we can conclude as in the proof of Lemma~\ref{1_wQs} that $\wQs(\mR^m)\geq \Q(\mS^m)$.
Let $\tilde{\phi}$ be a Killing spinor to the Killing constant $-\frac{1}{2}$ normalized such that $\Vert \tilde{\phi}\Vert_{L^q}=1$. Then, $D\tilde{\phi}=\lm (\mS^m) |\tilde{\phi}|^{q-2}\tilde{\phi}$. Using stereographic projection we obtain a smooth spinor $\phi= \left(\frac{1+r^2}{2}\right)^{\frac{-m+1}{2}}\tilde{\phi}$ on $\R^m$ with $L^q$-norm one and which satisfies $D\phi=\lm (\mS^m) |\phi|^{q-2}\phi$. Moreover,

\[\int_{\R^m} |\phi|^2\, \vo_E= C\int_{\R^m} \left(\frac{1+r^2}{2}\right)^{-m+1} |\phi|^2\, \vo_E= C' \int_0^\infty \left(\frac{1+r^2}{2}\right)^{-m+1} r^{1-m}\, \d r.\]
Thus, $\phi\in L^2(\R^m)$ for $m\geq 3$. Hence, $\wQs(\mR^m)= \Q(\mS^m)$ for $m\geq 3$.

An analogous argumentation for nonnegative solution $v\in C^2\cap L^\infty\cap L^2$ on $\R^m$ satisfying $Lv=\mu v^{p-1}$ and $\Vert v\Vert_{L^p}=1$ for a $\mu>0$ gives a nonnegative solution $\tilde{v}$ of the corresponding nonlinear eigenvalue equation on the sphere with $\int_{U_\epsilon (N)} \frac{1}{\rho^2}\tilde{v}^2 \,\vo_{\sigma^m}<\infty$. By regularity $\tilde{v}\in L^p\cap H_2^{p^*}$ with $p^*=\frac{2m}{m+2}$. Thus, by the Sobolev embedding theorem $\tilde{v}\in H_1^2$. Thus, similar as in Lemma~\ref{1_wQs} we see that $\wQ(\R^m)\geq \Q(\mS^m)$. Moreover, by \cite[Thm.~5]{farid} $\tilde{v}$ is continuous and everywhere positive.  Hence, we can estimate
\begin{align*} \int_{\R^m} {v}^2 \,\vo_{E}=&  \int_{\R^m} \left(\frac{1+r^2}{2}\right)^{-m+2}\tilde{v}^2 \,\vo_{E}\geq C \int_0^\infty  \left(\frac{1+r^2}{2}\right)^{-m+2} r^{m-1}\, \d r\\
\geq &C' \int_0^\infty  r^{-m+3}\, \d r.
\end{align*}
Thus, for $m= 3,4$ the solution $v$ was not in $L^2$ which contradicts the assumption. Hence, $\wQ(\mR^m)=\infty$ for $m=3,4$. 
For $m\geq 5$, we see with an analogous calculation that taking the constant solution $v$ of $Lv=\Q(\mS^m)v^{p-1}$, $\Vert v\Vert_{L^p}=1$ on the sphere, we obtain via stereographic projection a solution $\tilde{v}$ on $\R^m$ which is even in $L^2(\R^m)$. Thus, $\wQ(\mR^m)=\Q (\mS^m)$ for $m\geq 5$.
\end{proof}

\begin{example} 
Let $\phi$ be a Killing spinor on $\mS^2$ with $L^{q=4}$-norm one. Then, $D\phi=\Q (\mS^2)|\phi|^{2}\phi$.  We consider the three-branched covering $h\colon\mS^2\to \mS^2$, $z\mapsto z^3$. This map preserves the spin structure. Thus, we can pullback $\phi$ via 
$h$ and obtain a spinor $\tilde{\phi}$ on $\mS^2$, cp. \cite[Sec.~4]{Amm03a} fulfilling $D\tilde{\phi}=\Q (\mS^2) |\tilde{\phi}|^{2}\tilde{\phi}$ and  $\Vert \tilde{\phi}\Vert_{L^4}^4=3$. In particular, $\tilde{\phi}$ has zeros on the north and the south pole of $\mS^2$. Setting $\widehat{\phi}=\left(\frac{1}{3}\right)^{1/4} \tilde{\phi}$ we obtain $D\widehat{\phi}=3^{1/2}\Q (\mS^2) |\widehat{\phi}|^{2}\widehat{\phi}$ and  $\Vert \widehat{\phi}\Vert_{L^4}=1$. Using stereographic projection we obtain a spinor $\psi=(\frac{2}{1+r^2})^{1/2} \widehat{\phi}$ on $\R^2$ ($r$ being the radial coordinate in $\R^2$) with $D{\psi}=3^{1/2}\Q (\mS^2) |{\psi}|^{2}{\psi}$ and  $\Vert {\psi}\Vert_{L^4}=1$. Moreover, since $\widehat{\phi}$ vanishes at the north pole $N$, $|\widehat{\phi}(x)|\leq C\rho$ on $U_\epsilon(N)$ where $\rho=\dist(.,N)$.   by the estimate
\begin{align*}
 \int_{\R^2\setminus B_r(0)} |\psi|^2 \vo_{E} =  \int_{U_{\epsilon(r)}(N)} \frac{1+r^2}{2} |\widehat{\phi}|^2 \vo_{\sigma^2} \leq  
 C' \int_0^\epsilon \rho^2\frac{1+r^2}{2}  \rho\, d\rho= C'\int_0^\epsilon \frac{\rho^3}{\sin^2 \rho}\, \d \rho.
\end{align*}
Thus, $\psi\in L^2(\mR^2)$ and $\wQs(\R^2)\leq 3^{1/2}\Q (\mS^2)$.
\end{example}

 Summarizing we obtained for the spinorial invariants
 
 \begin{corollary}
  We have \[\wLas_{m,m-1}=\Las_{m,m-1}=\Q(\mS^m)\quad  \text{for\  all\ } m\geq 3\] and $3^{1/2}\Q (\mS^2) \geq \wLas_{2,1} =\wQs (\R^2)\geq \Q(\mS^2)=\Las_{2,1}$.
Moreover,
  \begin{align*}\wLa_{m,m-1}=&\La_{m,m-1}=\Q(\mS^m)\quad \text{\ for\ all\ }m\geq 5,\\
  \infty= \wLa_{m,m-1} >& \La_{m,m-1}=\Q(\mS^m) \quad \text{\ for\ }m=3,4.\end{align*}
 \end{corollary}

\section{Hijazi inequalities}\label{sec_hijazi}

On a closed spin manifold $(M^m, g)$, the Hijazi inequality provides a lower bound of the lowest eigenvalue $\lambda_0^2(g)$ of the square of the Dirac operator by the lowest eigenvalue of the conformal Laplacian $\mu(g)$, \cite[Thm.~A]{hijazi},
\begin{equation} \lambda_0^2(g)\geq \frac{m}{4(m-1)} \mu(g).\label{metr_hij}\end{equation}
Taking the infimum over all metrics conformal to $g$ with constant volume, one obtains the conformal Hijazi inequality \cite{NG11}
\begin{equation} \Qs(M)\geq \Q(M).\label{conf_hij}\end{equation}

We call \eqref{metr_hij} the metric Hijazi inequality and \eqref{conf_hij} the conformal Hijazi inequality. 
In this section, we want to discuss whether similar inequalities also hold on noncompact manifolds. In this context one should replace  the lowest eigenvalues in \eqref{metr_hij} by the infimum of the corresponding spectra whereas \eqref{conf_hij} remains unchanged. 

In \cite[Thm.~1.1 and 1.2]{NG11} the metric Hijazi inequality  was shown by the second author for complete spin manifold of finite volume fulfilling one of the following conditions:
\begin{enumerate}
 \item The infimum of the spectrum of the squared Dirac operator is an eigenvalue.
 \item The infimum of the spectrum of the squared Dirac operator is in the essential spectrum, $m\geq 5$ and the scalar curvature is bounded from below.
\end{enumerate}

In particular, this already  implies the conformal Hijazi inequality  for manifolds which admit a conformal metric $\bar{g}$ that is complete and of finite volume and where zero is not in the essential spectrum of the Dirac operator for $\bar{g}$ or where the second condition from above is fulfilled, cf. \cite[Thm.~1.3]{NG11}.

There are also examples of manifolds of bounded geometry where the metric Hijazi inequality does not hold. The simplest example is the hyperbolic space $\mH^m$ where $0$ is in the spectrum of the Dirac operator and the spectrum of the conformal Laplacian is $[\mu, \infty)$ with  $\mu=4\frac{m-1}{m-2}\frac{(m-1)^2}{4}-m(m-1)=\frac{m-1}{m-2}>0$. 

On the other hand, the hyperbolic space is conformal to a subset of the standard sphere. Thus, $\Qs(\mH^m)=\Qs(\mH^m)=\Qs(\mS^m)=\Q(\mS^m)$, see Lemma~\ref{1_QS} for details. Unfortunately it is still unclear whether the conformal Hijazi inequality \eqref{conf_hij} holds for all complete Riemannian spin manifolds.

In this section we prove  slightly modified conformal Hijazi inequalities. Some inequalities are proven only for the model spaces, some on more general manifolds, e.g. for manifolds of bounded  geometry with uniformly positive scalar curvature.

\begin{prop}\label{spin-eq2yamabe-ineq}
 Let $(M^m,g)$ be of bounded geometry with $m\geq 3$.
 Let $\phi\in L^q\cap C^0$ and $\lambda\in \mR$ with 
$D\phi=\lambda |\phi|^{\frac{2}{m-1}}\phi$  weakly and 
$\Vert \phi\Vert_{q}=1$ where $q=\frac{2m}{m-1}$.
Then $u:=|\phi|^{\frac{m-2}{m-1}}$ satisfies
\begin{equation}\label{yamabe-ineq}
   L u \leq 4\frac{m-1}m \la^2 u^{\frac{m+2}{m-2}} 
\end{equation}
in the sense of distributions. Moreover, the equation holds classically outside the zero-set of $u$.
\end{prop}

\begin{proof}By Corollary~\ref{improved-reg} and Lemma~\ref{bdd-sol}
$\phi\in L^\infty\cap C^2$.  We use the idea of Christian B\"ar and Andrei Moroianu written down in \cite[Prop.~3.4]{CGH}.
Let $\alpha:=\frac{m-2}{m-1}$ and $f:=\lambda |\phi|^{q-2}$. We define the Friedrich connection $\nabla^f$ as 
$\nabla^f_X \psi=\nabla_X\psi + \frac{f}{m} X\cdot \psi$. Then
for all points where $\phi\ne 0$ we estimate
  \begin{align*}
  \d^*&\d|\phi|^{\alpha}=\frac{\alpha}{2}|\phi|^{\alpha-2}\d^*\d|\phi|^2-\alpha(\alpha-2)|\phi|^{\alpha-2}|\d |\phi||^2\\
  =&\alpha|\phi|^{\alpha-2}\left(\langle \Delta \phi,\phi\rangle - |\nabla \phi|^2  -(\alpha-2)|\d |\phi||^2\right)\\
  =&\alpha|\phi|^{\alpha-2}\left(\langle \Delta \phi,\phi\rangle - |\nabla^f \phi|^2 - 2\frac{f}{m}\langle D\phi,\phi\rangle + \frac{f^2}{m}|\phi|^2  - (\alpha-2)|\d |\phi||^2 \right)\\
  =&\alpha|\phi|^{\alpha-2}\left(\langle D^2 \phi,\phi\rangle -\frac{\scal_M}{4}|\phi|^2- |\nabla^f \phi|^2 - 2\frac{f}{m}\langle {D\phi},\phi\rangle + \frac{f^2}{m}|\phi|^2- (\alpha-2)|\d |\phi||^2 \right)\\
  =&\alpha|\phi|^{\alpha-2}\left(\langle (D-f)^2 \phi,\phi\rangle -\frac{\scal_M}{4}|\phi|^2- |\nabla^f \phi|^2 + 2\frac{m-1}{m}f\langle {(D-f)\phi},\phi\rangle + \frac{m-1}{m}f^2|\phi|^2\right.  \\
&\left.+ \frac{m}{m-1}|\d |\phi||^2 \right)
  \end{align*}
  where we used in the last step that $\< [D,f]\phi,\phi\>=\<\d f\cdot \phi, \phi\>$ has to vanish since $\<\d f\cdot \phi, \phi\>_x\in \rm{i}\R$ but all the other terms are real. By Lemma~\ref{improved-reg}, $\phi$ is in $C^2$ and  all equations above hold in the classical sense. In particular $(D-f)\phi=0$. As the spinor~$\phi$ is in the kernel of the operator $D-f$ we can use the refined Kato inequality $|\nabla^f \phi|^2\geq \frac{m}{m-1} |\d |\phi||^2$, see \cite[(3.9)]{CGH}.  Thus, we get for the $u$ defined in the proposition
\begin{align*}
  \d^*\d u & \leq -\alpha\frac{\scal_M}{4}u + \alpha \frac{m-1}{m}f^2u.
\end{align*}
Thus, using $a=\frac{4}{\alpha}$ this means that
\begin{align*}
  L u = a\Delta u+ \scal_Mu& \leq 4 \frac{m-1}{m}f^2u = 4 \frac{m-1}{m}\lambda^2u^{\frac{m+2}{m-2}}.
\end{align*}
As remarked above this all holds outside of the zero set of $u=|\phi|^{\frac{m-2}{m-1}}$. From Corollary~\ref{viscosity-implies-distrib_cor} we see that inequality 
\eqref{yamabe-ineq} holds distributionally since $u$ is a nonnegative function.
\end{proof}

\begin{proposition}\label{prop_hij_scal_pos}
We assume the conditions of Proposition~\ref{spin-eq2yamabe-ineq}. Additionally we assume
that $\scal_M\geq s_0>  0$ or $\phi\in L^{2\frac{m-2}{m-1}}$, then 
\begin{equation}\label{hijazi-scal-pos}
  \Q(M,g)\leq 4 \frac{m-1}m \lambda^2.
\end{equation}
\end{proposition}

\begin{proof}Let $u$ be  defined as in the previous proposition. 
For any regular value $\epsilon>0$ of $u$,  we consider $V_\epsilon:=\{u\geq \epsilon\}$. Note that by Lemma~\ref{0-bdd-sol} $\lim_{x\to \infty} u(x)=0$, and hence $(V_\epsilon)_\epsilon$ exhausts $M$ as $\epsilon\to 0$. Let $\nu$ be the exterior unit normal field of the boundary of $V_\epsilon$. Then, $\partial_\nu u\leq  0$. Thus,
integration over 
$V_\epsilon$ of Inequality~\eqref{yamabe-ineq} multiplied by $u$ gives 
  \begin{align}
  4 \frac{m-1}{m}\lambda^2 \underbrace{\int_{V_\epsilon} u^{\frac{2m}{m-2}}\vo_g}_{\to 1\text{\ as\ } \epsilon\to 0}
& \geq \int_{V_\epsilon} \left(a u\, \d^*\d u + \scal_M u^2\right)\vo_g\nonumber\\
& =  \int_{V_\epsilon} \left(a |\d u|^2 + \scal_M u^2\right)\vo_g - \int_{\partial V_\epsilon}a u\partial_\nu u \,\vo_g\nonumber\\
& \geq  \int_{V_\epsilon} \left(a |\d u|^2 + \scal_M u^2\right)\vo_g\label{hij_inequ}.
\end{align}
In the case that $\scal_M\geq s_0> 0$ this implies $u\in H_1^2(M)$. 

For the remaining case that $\phi\in L^{2\frac{m-2}{m-1}}$ we have $u\in L^2$. Thus, we obtain 
 \begin{align*}
   \lim_{\epsilon\to 0} \int_{V_\epsilon} |\d u|^ 2\vo_g & \leq 
4 \frac{m-1}{ma}\lambda^2  +\frac{\sup_M |\scal_{M}|}{a}  \|u\|_{L^2}^2,
\end{align*}
and this as well implies $u\in H_1^2$.

Sard's theorem tells us that the set of regular $\epsilon$ is dense. In the limit $\epsilon\to 0$
we then get \eqref{hijazi-scal-pos} from \eqref{hij_inequ}.
\end{proof}

\begin{example}[Spherical cap solution on hyperbolic space]\label{sph_cap} Let $B_r\subset \mS^m$ be a ball in the standard sphere of radius $r$. Let $\phi$ be a Killing spinor on the sphere with Killing constant $-\frac{1}{2}$ normalized as $|\phi|=\vol (B_r)^{-\frac{1}{q}}$ for $q=\frac{2m}{m-1}$. Then $\Vert \phi\Vert_{L^p(B_r)}=1$ and $D^{\mS^m}\phi=\frac{m}{2}\phi=\frac{m}{2}\vol (B_r)^{\frac{q-1}{q}} |\phi|^{q-2}\phi$. Let $u\colon \mH^m\to B_r$ be a conformal map from the hyperbolic space to the spherical cap such that $g_{\mH^m}=f^2 u^* \sigma^m$. Then, using identification of the spinor bundles, as in Example \ref{no_Linfty_bound} and setting $\tilde{\phi}:=f^{-\frac{m-1}{2}}\phi$ we get by conformal invariance that 
\[D^{\mH^m}\tilde{\phi}=\underbrace{\frac{m}{2}\vol (B_r)^{\frac{q-1}{q}}}_{=:\lambda_r} |\tilde{\phi}|^{q-2}\tilde{\phi}\ \text{ and\ }  \Vert \tilde{\phi}\Vert_{L^q(\mH^m)}=1\ \text{ and\ }  \Vert \tilde{\phi}\Vert_{L^\infty(\mH^m)}<\infty.\]
 Then $\lambda_r\to 0$ as $r\to0$. Nevertheless, $\Q(\mH^m)=\Q(\mS^m)$. Thus, Proposition~\ref{prop_hij_scal_pos} does not hold without the assumption $\scal_M\geq s_0>  0$ or $\phi\in L^{2\frac{m-2}{m-1}}$. Hence, the conformal Hijazi inequality $\wQs(\mH^m)\geq \Q(\mH^m)$ which trivially follows from Lemma~\ref{1_wQs} is no longer true if we remove the $L^2$-condition in the definition of $\wQs$.
\end{example}

 We now want to use these inequalities to prove Hijazi inequalities for the  model spaces $\Mc^{m,k}$. In this goal we will examine whether for certain $m$ and $k$  there is a spinor~$\phi\in L^{2\frac{m-2}{m-1}}$ satisfying the assumptions of Proposition~\ref{spin-eq2yamabe-ineq}.
\begin{prop}\label{add_reg_MC} Let $m\geq 3$. Let  $0\leq k< m-2$, $c\in [0,1]$ or $k=m-2$ and $c\in [0,1)$.
On the manifold 
$\Mc^{m,k}$, 
we consider a spinor field $\phi\in L^q\cap C^0$ solving 
\[D\phi=\lambda |\phi|^{\frac{2}{m-1}}\phi, \qquad \Vert \phi\Vert_{q}=1\]
for $\lambda\in \mR$ and $q=\frac{2m}{m-1}$. 
Then, $\phi\in L^{2\frac{m-2}{m-1}}$.
\end{prop}

We also now that $\phi\in L^\infty$ by Lemma~\ref{bdd-sol}.

\begin{lemma}\label{aux_lem}
Under the assumptions of the proposition we have $(m-2)(m-k-1)>ck$, unless
$k=m-2$ and $c=1$.
\end{lemma}

\begin{proof}[Proof of the lemma]
The condition $(m-2)(m-k-1)>ck$ is equivalent to  
$(m-1)(m-k-2)>-(1-c)k$.
\end{proof}

\begin{proof}[Proof of Proposition~\ref{add_reg_MC}]
 By Proposition~\ref{inv_Mc} $D$ is $L^r$-invertible if 

\begin{equation}\label{D.invert.cond}
\frac{m-k-1}{2}> ck\left| \frac{1}{r}-\frac{1}{2}\right|. 
\end{equation}
By assumption we have $D\phi = \lambda |\phi|^{2/(m-1)}\phi\in L^{\frac{2m}{m+1}}$. 
The condition~\eqref{D.invert.cond} for $r:=2m/(m+1)$ is equivalent to 
$m(m-k-1)>ck$ which is fulfilled by assumption. Thus, we obtain $\phi\in L^\frac{2m}{m+1}$. Hence,
\[D\phi=\lambda |\phi|^{q-2}\phi \in L^{\frac{2m}{(m+1)(q-1)}=\frac{2m(m-1)}{(m+1)^2}}.\] 
Note that for $m\geq 3$ we have $\frac{2m(m-1)}{(m+1)^2}\leq 2\frac{m-2}{m-1}=: s$. 
Hence, using $\phi$ in $L^\infty$ we get $D\phi\in L^s$. Moreover, $D$ is $L^s$-invertible as condition~\eqref{D.invert.cond} for $r:=s$ is equivalent to 
$(m-2)(m-k-1)>ck$ which is provided by assumption and Lemma~\ref{aux_lem}. Thus $\phi\in L^s$. 
\end{proof}

\begin{example}
In the exceptional case $k=m-2$ and $c=1$ the conclusion of Proposition~\ref{add_reg_MC} is not correct.
To see this, we construct the following example. We consider a Killing spinor
on $\mS^m$ and transport it conformally to $\mathbb{M}_1^{m,m-2}$, similar to the proof of Lemma~\ref{1_wQs}.  
The spinor falls off as $e^{-(m-1)r/2}$ where $r$ is the distance to a 
fixed point on $\mH^{m-1}$ as introduced in Section \ref{mod_sp}.
Then the $L^{2\frac{m-2}{m-1}}$-norm of $\phi$ is infinite.
A similar example is provided by Example~\ref{sph_cap} in the case $k=m-1$ and $c>0$.
\end{example}

\begin{corollary}\label{hijazi_inq}
Consider $\Mc^{m,k}$ with $m\geq 3$. Let  $0\leq k< m-2$, $c\in [0,1]$ or $k=m-2$ and $c\in [0,1)$.
Let $\phi\in L^q\cap C^0$ be a solution of 
$D\phi=\lambda |\phi|^{\frac{2}{m-1}}\phi$ on $\Mc^{m,k}$ with $\lambda\in \mR$ and 
$\Vert \phi\Vert_{q}=1$ where $q=\frac{2m}{m-1}$.
Then, 
\begin{equation*}
\lambda^2\geq \frac{m}{4(m-1)}\Q(\Mc^{m,k}).
\end{equation*}
\end{corollary}

\begin{proof} We set $u=|\phi|^{\frac{m-2}{m-1}}$. 
 By Proposition~\ref{add_reg_MC} $u\in L^2$. Then, Proposition~\ref{prop_hij_scal_pos} gives the corollary.
\end{proof}

\begin{corollary}[Conformal Hijazi inequality for the model spaces]\label{Q_leq_wQs}
For $0\leq k\leq m-2$ and $c\in[0,1]$ we have
   \[\wQs(\Mc^{m,k})\geq \Q(\Mc^{m,k}).\]
In particular, $\wLas_{m,k}\geq \sLa_{m,k}$ for $0\leq k\leq m-2$.
   \end{corollary}

\begin{proof}  For $k<m-2$ or $k\leq m-2$ and $c<1$ this follows immediately from Corollary~\ref{hijazi_inq} and the definition of $\wQs$. The remaining case, $k=m-2$ and $c=1$, was treated in Lemma~\ref{1_wQs}.
\end{proof}

\begin{remark} In the case $k=m-2$ we obtain together with \cite[Lem.~3.8]{ADH} that
  $\wQs(\Mc^{m,m-2})\geq \Q(\Mc^{m,m-2})\geq  c^\frac{2}{m} \Q(\mS^m)$.
  For a test function $v\in C^\infty(\Mc^{m,m-2})$ that is constant along the $S^1$  one can calculate (since the scalar curvature of $S^1$ is zero) that 
  \[\mathcal{F}^{\Mc^{m,m-2}}(v)= c^\frac{2}{m}\mathcal{F}^{\mathbb{M}_1^{m,m-2}}(v).\] Since  $\Q(\mathbb{M}_1^{m,m-2})=\Q(\mS^m)$ is minimized by a $v$ that is constant along $S^1$, we have $\Q(\Mc^{m,m-2})=  c^\frac{2}{m} \Q(\mS^m)$. Thus, together we obtain
  $$\wQs(\Mc^{m,m-2})\geq \Q(\Mc^{m,m-2})=  c^\frac{2}{m} \Q(\mS^m).$$ 
  In particular, $\sLa_{m,m-2}=0$.
\end{remark}

\section{Minimizer of the variational problems}

The Euler-Lagrange equations of the constants $\Q$ and $\lm$ defined via functionals read as
\[ L u=\Q u^{p-1}\ \text{with\ } \Vert u\Vert_{p=\frac{2m}{m-2}}=1\]
and 
\[ D\phi = \lm |\phi|^{q-2}\phi\ \text{with\ } \Vert \phi\Vert_{q=\frac{2m}{m-1}}=1.\]
Assume now such minimizing solutions $u\in H_1^2 \cap L^\infty$ and $ \phi\in  H_1^{\frac{2m}{m+1}}\cap L^\infty$ exist. Then, we also have $u\in C^2$ and $\phi\in C^1$. Then, $\wQ\leq \Q$. Moreover, by interpolation $\phi\in L^2$ and, thus, $\wQs \leq \Qs$.\\

We now recall some theorems for the existence of such solutions on {\it almost homogeneous} manifolds $(M,g)$, i.e., on Riemannian manifolds on which there is a relatively compact
 set $U \subset\subset M$ such that for all $x \in M$ there is an isometry $f\colon M \to M$ with $f (x) \in U$. Note that a manifold is almost homogeneous if and only if the isometry group $G={\rm Isom}(M,g)$ acts cocompactly on $M$. This follows since the distance between the orbits on $M$ defines a metric on $M/G$, and the induced topology is the quotient topology of $\pi\colon M\to M/G$. 

\begin{theorem}\cite[Thm.~13]{NG3}\label{NG3}
 Let $(M^m , g)$ be a Riemannian manifold of bounded geometry with $\scal_M \geq {\rm const} > 0$ for a constant
$c$ and $ \Q(M, g)< \Q(\mS^m)$. Let $(M,g)$ be almost homogeneous. 
Then, there is a positive smooth solution $u\in H_1^2 \cap L^\infty\cap C^2$ of the Euler-Lagrange equation
$Lu = \Q(M)u^{\frac{m+2}{m-2}}$ and $\Vert u \Vert_{\frac{2m}{m-2}} = 1$.
\end{theorem}
In the reference of this reference, $u\in C^2$ was not explicitly stated, but follows from standard elliptic regularity theory.

\begin{cor}\label{wQ_leq_Q}
 Let $(m-k-1)(m-k-2)>c^2(k+1)k$, $0\leq c<1$ or let $k\leq m-3$ and $c=1$. 
 Then, 
\[\wQ(\Mc^{m,k})\leq \Q(\Mc^{m,k}).\]
\end{cor}

\begin{proof} For, $k\leq m-3$ and $c=1$ we even have equality by Lemmata~\ref{1_QS} and~\ref{1_wQ}.
For $k=0$, $\Mc^{m,0}$ is diffeomorphic to $\mathbb{M}_1^{m,0}$. Thus, by Lemma~\ref{1_wQ} we even have equality. Let now
 $k>0$. Then the condition $c<1$ implies that $\Mc^{m,k}$ is not conformally flat. In the case $m\geq 6$ Aubin's inequality, cp. Remark~\ref{inv_sphere}, provides $\Q(\Mc^{m,k})<\Q(\mS^m)$. For $m<6$ we have $\Q(\Mc^{m,k})<\Q(\mS^m)$ by the following theorem. Thus Theorem~\ref{NG3} implies the existence of a solution $u$ as above, which directly implies the corollary.
\end{proof}

\begin{theorem}\cite[Cor.~1]{ammann.grosse:p15}
 Let $m=n+k+1$, $m\geq 3$, $k>0$,  and $c\in [0,1)$. Then
\[\Q(\mS^n\times \mH_c^{k+1},\sigma^n+g_c)<\Q(\mS^m, \sigma^m).\]
\end{theorem}

\begin{theorem}\cite[Thm.~16]{NG12}\label{NG12}
 Let $(M^m , g)$ be a Riemannian spin manifold of bounded geometry with $\scal_M \geq C > 0$ for a constant
$C$ and $ \lm(M, g)< \lm(\mS^m)$. Let $(M,g)$ be almost homogeneous.
Moreover, assume that the Dirac operator $D$ on $M$ is invertible 
as an operator from $L^s$ to $L^s$ for $s=\frac{2m}{m+1}$.
Then, there is a positive smooth solution $\phi\in H_1^{\frac{2m}{m+1}} \cap L^\infty\cap C^1$ 
of the Euler-Lagrange equation $D\phi = \lm |\phi|^{\frac{2}{m-1}}\phi$ and $\Vert \phi \Vert_{\frac{2m}{m-1}} = 1$.
\end{theorem}

In the reference of this theorems $\phi\in L^\infty$ was not stated explicitly, but can be seen directly from the proof in \cite{NG12} or alternatively by Lemma~\ref{bdd-sol}. Moreover, in the original version of  Theorem~\ref{NG12} it was requested that $D$ is invertible for all $s\in [\frac{2m}{m+1}, \frac{2m}{m+1}+\epsilon]$ for some $\epsilon>0$. But if $D$ is invertible for $s=\frac{2m}{m+1}$, then it is also invertible for the conjugate exponent $\frac{2m}{m-1}$ and by interpolation for all $s$ in between, cp. \cite[App.~B]{ammann.grosse:p13a}.\\

For the special case of manifolds that are product spaces $M=N_1\times N_2$ of an almost homogeneous manifold $N_1$ and a closed manifold $N_2$ one can relax the assumption on the positive scalar curvature in Theorem~\ref{NG12}:

\begin{theorem}\label{product_NG12}
 Let $(M^m=N_1\times N_2 , g)$ be a Riemannian spin manifold that is a product manifold of an almost homogeneous manifold $N_1$ and a closed manifold $N_2$. Let $\lambda_{N_2}^2$ be the lowest eigenvalue of the square $(D^{N_2})^2$ of the Dirac operator on $N_2$,  and let $\lambda_{N_2}^2+ \frac{1}{4} \inf \scal_{N_1}=:c>0$. Moreover, let $ \lm(M, g)< \lm(\mS^m)$, and assume that the   Dirac operator $D$ on $M$ is invertible as an operator from $L^{q^*}$ to $L^{q^*}$ for $q^{*}=\frac{2m}{m+1}$.
Then, there is a positive smooth solution $\phi\in H_1^{\frac{2m}{m+1}} \cap L^\infty\cap C^1$ 
of the Euler-Lagrange equation $D\phi = \lm |\phi|^{\frac{4}{m-1}}$ and $\Vert \phi \Vert_{\frac{2m}{m-1}} = 1$.
\end{theorem}

\begin{proof}
We start as in the proof of the general result in  \cite{NG12} - here  we shortly recall the steps that remain the same: For that let $\rho$ be a radial admissible weight, see \cite[Sect.~A.1]{NG12}, $\rho\leq 1$. Then, by \cite[Lem.~13]{NG12}  we obtain for each $s\in [2, q=\frac{2m}{m-1})$ a sequence $\phi_s\in H_1^{s^*}\cap C^1$ with $D\phi_s=\lambda^{\alpha}_s \rho^{\alpha s} |\phi_s|^{s-2}\phi_s$ and $\Vert \rho^\alpha \phi_s\Vert_{L^s}=1$ where $s$ and $s^*$ are conjugate, $\alpha=\alpha(s)\to 0$  as $s\to q$ and $\mu:=\limsup_{s\to q} \lambda^\alpha_s\leq \lm$. Then by \cite[Lemmata~14 and 15]{NG12} a subsequence $\phi_s=\phi_{\alpha(s),s}$ converges to a function $\phi\in H_1^{s^*}$ in $C^1$-topology on each compact subset, and we have $D\phi=  \mu |\phi|^{q-2} \phi$. It remains to show that $\Vert \phi\Vert_{L^{q}}=1$, i.e., in particular that $\phi$ is nonzero. Then the arguments that $\phi$ is a solution as desired just follow again the lines of \cite[
Lemma~15]{NG12}.

We prove the remaining point by contradiction, i.e., we assume that $\phi=0$: Note that  by \cite[Lem.~33]{NG12} for each $s$ we have  $\lim_{|x|\to \infty} |\phi_s|=0$. Thus, we can fix $x_s\in N_1$ such that $\int_{\{x_s\}\times N_2} |\phi_s|^2$ attains its maximum. By the almost-homogeneity of $N_1$ we can assume that all $x_s$ are contained in a compact subset $K$ of $M$.   
Moreover, then 
\begin{align*}
0\leq& \Delta_{N_1} \int_{\{x_s\}\times N_2} |\phi_s|^2= \int_{\{x_s\}\times N_2} \Delta_{N_1}|\phi_s|^2\\
=&2 \Re \int_{\{x_s\}\times N_2} \< \nabla^*_{N_1}\nabla_{N_1}\phi_s , \phi_s\>-2 \int_{\{x_s\}\times N_2} |\nabla^{N_1}\phi_s|^2\leq 2\Re \int_{\{x_s\}\times N_2} \< \nabla^*_{N_1}\nabla_{N_1}\phi_s , \phi_s\>.
\end{align*}

Using $\rho\leq 1$, $D\phi_s=\lambda^{\alpha}_s \rho^{\alpha s} |\phi_s|^{s-2}\phi_s$ and $\< \d(\rm function)\cdot \psi, \psi\>_x\in {\rm i}\mR$ we get
\[\Re \int_{\{x_s\}\times N_2} \<D^2 \phi_s, \phi_s\> \leq (\lambda^\alpha_s)^2 \int_{\{x_s\}\times N_2} |\phi_s|^{2(s-1)}.\]

The square of the Dirac operator decomposes as $D^2=(D^{N_1})^2+(\tilde{D}^{N_2})^2$ where  $\tilde{D}^{N_2}={\rm diag} (D^{N_2}, -D^{N_2})$ in case that both manifolds are odd dimensional and $\tilde{D}^{N_2}=D^{N_2}$ else, see e.g. \cite[Sect.~2.5]{ammann.grosse:p13a}.  The operators $(\tilde{D}^{N_2})^2$ and $({D}^{N_2})^2$ have the same spectrum. 
Then, together with the Schr\"odinger-Lichnerowicz formula we obtain

\begin{align*} \int_{\{x_s\}\times N_2}\!\!\!|\tilde{D}^{N_2}\phi_s|^2 + \Re \int_{\{x_s\}\times N_2}\!\!\!\< \nabla^*_{N_1}\nabla_{N_1} \phi_s, \phi_s\> +\frac{\scal_{N_1}}{4}\int_{\{x_s\}\times N_2}\!\!|\phi_s|^2\\
\leq (\lambda^\alpha_s)^2  \int_{\{x_s\}\times N_2}\!\!\!|\phi_s|^{2(s-1)}\end{align*}

and, thus,

\[\left(\lambda_{N_2}^2 +\frac{\scal_{N_1}}{4}\right)\int_{\{x_s\}\times N_2} |\phi_s|^2   \leq (\lambda^\alpha_s)^2 \max_{\{x_s\}\times N_2} |\phi_s|^{2(s-2)} \int_{\{x_s\}\times N_2} |\phi_s|^2.\]

Hence,

\begin{align*}\left(\lambda_{N_2}^2 +\frac{\scal_{N_1}}{4}\right) (\lambda^\alpha_s)^{-2}  \leq&  \max_{\{x_s\}\times N_2} |\phi_s|^{2(s-2)}\\
c(\lm)^{-2}\leq \liminf_{s\to q}\left(\lambda_{N_2}^2 +\frac{\scal_{N_1}}{4}\right) (\lambda^\alpha_s)^{-2}  \leq& \liminf_{s\to q} \max_{\{x_s\}\times N_2} |\phi_s|^{2(s-2)}.
\end{align*}

Note that all $x_s$ are contained in a compact set. Thus, a subsequence of $x_s$ converges to some $x\in M$.
But, $\phi=0$ means that then the right hand-side is zero which gives the desired contradiction.
\end{proof}

\begin{cor}\label{wQs_leq_Qs}
 Let $(m-k-1)^2>c^2(k+1)k$
 and $\Qs (\Mc^{m,k})< \Qs (\mS^m)$. 
Then, $\wQs (\Mc^{m,k})\leq \Qs(\Mc^{m,k})$.
\end{cor}

\begin{proof} This corollary follows from Theorem~\ref{product_NG12}. Set $N_1=\mH_c^{k+1}$ and $N_2=\mS^{m-k-1}$. Then $\lambda_{N_2}^2=\frac{(m-k-1)^2}{4}$ and $\scal_{N_1}=-c^2k(k+1)$.
\end{proof}

\begin{remark}
 Note that using Theorem~\ref{NG12} would lead to the condition $(m-k-1)(m-k-2)>c^2k(k+1)$. Thus, here Theorem~\ref{product_NG12} gives a better result.
\end{remark}

\section{\texorpdfstring{ $\Lambda$-invariants for $k=m-3$}{Lambda-invariants for k=m-3}}\label{km3}

For applying the surgery monotonicity formulas, cf. Theorems~\ref{surgthm} and~ \ref{surgspinthm}, one needs explicit positive lower bounds for $\La_{m,k}$ in dimension $k\leq m-3$ and for $\Las_{m,k}$ in dimension $k\leq m-2$. 
In the case $k\leq m-4$ and in  the case $k+3= m\leq6$ explicit positive lower bounds for $\La_{m,k}$ were provided in \cite{ammann.dahl.humbert:p12}. Using $\Las_{m,k}\geq \La_{m,k}$, cp. Corollary~\ref{Qs_leq_wQs}, Proposition~\ref{wLas_leq_Las} and \cite[Thm.~3.1 and Cor.~3.2]{ammann.dahl.humbert:p11b}, this yields explicit positive lower bounds for $\Las_{m,k}$ in these cases. However the techniques we have developed in the previous sections provide also explicit positive lower bounds for 
 $\sLa_{m,m-3}$ for any $m> 6$ which is the subject of the present section.
 
 Note that by Corollaries \ref{Q_leq_wQs} and~\ref{Qs_leq_wQs} and by Proposition~\ref{La_leq_wLa} we have
$\wLas_{m,m-3}=\sLas_{m,m-3}\geq \sLa_{m,m-3}$ and $\wLa_{m,m-3}\geq \sLa_{m,m-3}$. Thus, any lower bound for $\sLa_{m,m-3}$ is also a lower bound for the other three invariants. In Section \ref{surg_theorems} we pointed out that $\Las_{m,m-3}=\sLas_{m,m-3}$.  

As a first step we estimate $Q_c$:
It was derived in \cite[Thm.~4.1]{ammann.dahl.humbert:p12} 
for all $c\in (0,1)$ that
\begin{align} \label{est_general}
Q_c
\geq &
\left( 
\frac{Q_0}{Q_1} - 
\frac{c^2 (k+1)k}{(1-c^2)(m-k-1)(m-k-2) + c^2 (k+1)k}
\left(\frac{Q_0}{Q_1} -c^{\frac{2(m-k-1)}{m}}\right)
\right) Q_1 
\end{align}
where we defined $Q_c:=\Q(\Mc^{m,k})$.
On the other hand it follows for $m\geq 6$ and $k=m-3$ 
from  \cite[Sec.~4.1 (iv)]{ammann.dahl.humbert:p12} that
\begin{align*}
Q_0 \geq 
\frac{m a}{ 24^{\frac{3}{m}} ((m-3) a_{m-3})^{\frac{m-3}{m}}} 
\Q(\mS^{m-3})^{\frac{m-3}{m}} 
\Q(\mS^3)^{\frac{3}{m}}=: \widehat{Q}_0.
\end{align*}

Inequality~\eqref{est_general} reads for $k=m-3\geq 3$ as
\begin{align*} 
Q_c \geq &
\left( 
\frac{Q_0}{Q_1} - 
\frac{c^2 (m-2)(m-3)}{(1-c^2)2 + c^2 (m-2)(m-3)}
\left(\frac{Q_0}{Q_1} -c^{4/m}\right)
\right) Q_1\\
\geq &
\frac{(1-c^2)2\widehat{Q}_0 +c^{2+4/m} (m-2)(m-3)Q_1}{(1-c^2)2 + c^2 (m-2)(m-3)} 
=:L_m(c^2).
\end{align*}

Since $\widehat{Q}_0$ and $Q_1=\Q(\mS^m)$ are known explicitly one can compute the infimum of $L_m(c^2)$ numerically for fixed $m$, see Table~\ref{tab:codim3} for some explicit values.

\begin{table}
  \centering
  \begin{tabular}{c||c|c|c|c|c|c|c|c|c}
    $m$ &  7& 8& 9& 10 & 11&12&13&14&15\\
    \hline
    $\Q(\mS^m)$ & 
    113.5 & 130.7 & 147.88 & 165.0& 182.2 &199.3 &216.4&233.5& 250.6\\
    \hline
    $L_{m,m-3}$  & 
    65.2 & 78.7 & 91.8 & 104.9& 118.1 &131.5& 145.0 &158.6&172.4
  \end{tabular}
\\[1mm]
  \caption{Some explicit values for $L_{m,m-3}:=\inf_{c\in [0,1]} L_m(c^2)$ -- a lower bound for $\sLa_{m,m-3}$, $\wLa_{m,m-3}$ and $\Las_{m,m-3}$. The values are rounded -- $\Q(\mS^m)$ is rounded to the nearest multiple of $1/10$ and $L_{m,m-3}$ is always rounded down.}
  \label{tab:codim3}
\end{table}

For general $m$ we can still estimate the infimum of $L_m(c^2)$: 

\begin{align*}
\frac{\d}{\d s} L_m(s)=& \frac{-2\widehat{Q}_0 +(1+2/m) s^{2/m} (m-2)(m-3)Q_1}{2 + s((m-2)(m-3)-2)}\\ &-\frac{((1-s)2\widehat{Q}_0 +s^{1+2/m} (m-2)(m-3)Q_1)\left[(m-2)(m-3)-2\right] }{(2 + s((m-2)(m-3)-2))^2}. 
\end{align*}
The condition 
$\frac{\d}{\d s} L(s)=0$ is equivalent to 
\begin{align*}
 f(s):=\frac{\left(2 + s((m-2)(m-3)-2)\right)^2}{(m-2)(m-3)}\frac{\d}{\d s} L(s)= s^{\frac{2}{m}+1}A_0+s^{\frac{2}{m}}A_1+A_2=0
\end{align*}
where $A_0=\frac{2}{m}Q_1((m-2)(m-3)-2)$, $A_1=2(\frac{2}{m}+1)Q_1$, and $A_2=-2\widehat{Q}_0$. There is at least one zero in the interval $(0,1)$ since $f(0)<0$ and $f(1)>0$ (since $A_0>0$ and $Q_1\geq \widehat{Q}_0$). Let $F(u)=f(u^m)=u^{m+2}A_0+u^2A_1+A_2$, and let $u_0\in (0,1)$ be a zero of $F$. Then 
\begin{align*} 
 \frac{F(u)}{u-u_0}= u^{m+1}A_0+u^mA_0u_0+\ldots+uA_0u_0^m+uA_1+A_0u_0^{m+1}+A_1u_0.
\end{align*}
But for positive $u$ the last polynomial is always positive. Thus, there is only one zero of $f$ in the interval $(0,1)$.  

Moreover, $f(s)> s^{\frac{2}{m}}A_1+A_2$. Thus, $f(c_2^2)>0$ where $c_2:=\left(\frac{-A_2}{A_1}\right)^{\frac{m}{2}}= \left(\frac{m\widehat{Q}_0}{(m+2)Q_1}\right)^{\frac{m}{2}}$.  Hence, 

\begin{align*}
\sLa_{m,m-3}=\inf_{c\in [0,1]} Q_c \geq &  \inf_{c\in [0,c_2]} L_m(c^2)\\
 \geq &\inf_{c\in [0,c_2]} \frac{(1-c^2)2\widehat{Q}_0 +c^{4/m+2} (m-2)(m-3)Q_1}{2 + c_2^2 ((m-2)(m-3)-2)}. 
\end{align*}

Since $(1-c^2)2\widehat{Q}_0 +c^{4/m+2} (m-2)(m-3)Q_1$ attains its minimum for  $c_3^\frac{4}{m}=\frac{2m\widehat{Q}_0 }{(m+2)(m-2)(m-3)Q_1}$ we obtain

\begin{align*}
 \sLa_{m,m-3}\geq&\frac{(1-c_3^2)2\widehat{Q}_0 +c_3^{4/m+2} (m-2)(m-3)Q_1}{2+ c_2^2 ((m-2)(m-3)-2)}.
\end{align*}

Together with the explicit positive lower bounds 
$\sLa_{3,0}=\Q(\mS^3)=6\cdot 2^{2/3}\pi^{4/3}= 43,823233...$,
$\sLa_{4,1}> 38.9$ \cite[Example~4.7]{ammann.dahl.humbert:p11b}, and
$\sLa_{5,2}> 45.1$ \cite[Example~4.10]{ammann.dahl.humbert:p11b} 
we obtain explicit positive lower bounds for all $\sLa_{m,m-3}$.

These methods obviously also yield explicit positive lower bounds for all 
$\wQs (\mathbb{H}^{m-2}\times c\mS^2)$.

\section{Bordism arguments}\label{sec.bordism.arg}

\begin{proof}[ Proof of Proposition~\ref{sis.bound}]
By a theorem of Stolz, \cite[Thm.~B]{stolz}, $M$ is spin-bordant to the total space $M_0$ of an $\mH P^2$-bundle over a base $Q$ for 
which the structure group is $PSp(3)$. In particular, in dimension $m=5,6, 7$ this implies that $Q=\varnothing$. Each manifold $M$ in dimension $m=5,6,7$ is a spin boundary. 

Moreover, by the extended Stolz theorem \cite[Prop.~6.5]{ammann.dahl.humbert:p11b} we can assume that $Q$ is connected if $m\geq 9$ and that $Q$ is simply connected if $m\geq 11$. Thus, in dimension $m\geq 11$ also $M_0$ can be chosen to be connected and simply connected. In case that $Q=\varnothing$, set $M_1:=\mS^m$ otherwise $M_1:=M_0$. 

First, let $m=5,6,7$ or $m\geq 11$. As $M_1$ is simply connected  $M$ can be obtained from $M_1$ by a sequence of surgeries of dimensions $\ell$ where $2\leq \ell \leq m-3$, see \cite[Prop.~5.1]{ammann.dahl.humbert:p12}. Using Theorem~\ref{surgthm} this implies that $\sis(M)\geq \min\{ \sis(M_1), \Las_{m}\}$. 

In the case $m=5,6,7$ and in the case $m\geq 8$ and $Q=\varnothing$ we use $\Las_{m,k}=\sLas_{m,k}\leq \sis(\mS^m)$ for $k\leq m-2$, see Theorem~\ref{theo3.1}, and obtain
$\sis(M)\geq \Las_{m}$. For $m\geq 11$ and $Q\neq \varnothing$, i.e. $M_0=M_1$, we use the conformal Hijazi inequality for $M_1$ and  $\si(M_1)\geq \Q(\mH P^2\times \mathbb{R}^{m-8})$ for $m\geq 8$, see \cite{Streil}, where $\mH P^2\times \mathbb{R}^{m-8}$ carries the product metric of the standard metrics of both factors. Then we obtain 
\[ \sis(M)\geq \min\{ \Las_{m}, \Q(\mH P^2\times \mathbb{R}^{m-8})\}. \]

In dimension $m=8$ and $Q\neq \varnothing$, $M_1$ is a disjoint sum  of copies of $\mH P^2$, possibly with reversed  orientation. Using $0$-dimensional surgeries $M_1$ is spin bordant to the connected and simply connected manifold $M_2:=\mH P^2\# \ldots \# \mH P^2$, possibly with reversed orientation. Thus, $\sis(M_2)\geq \sis(M_1)=\sis (\mH P^2)\geq Q^*(\mH P^2)$. As above $M$ can then be obtained from $M_2$ by   surgeries of dimensions $\ell$ where $2\leq \ell \leq m-3$.
Using again Theorem~\ref{surgthm} this implies that $\sis(M)\geq \min\{\Las_{m}, \sis(M_2)\}\geq \min\{ \Las_m, Q^*(\mH P^2)\} $.

Let now $m=9,10$ and $Q\neq \varnothing$. Then, $M$ can be obtained from $M_1$ by surgeries of dimensions $\ell$ where $1\leq \ell \leq m-3$.
Using  Theorem~\ref{surgthm} this implies that $\sis(M)\geq \min\{\Las_{m,1}, \Las_{m}, \sis(M_1)\}$. Similar to above we get  \[\sis(M)\geq \min\{\Las_{m,1}, \Las_{m},  \Q(\mH P^2\times \mathbb{R}^{m-8})\}.\]
\end{proof}

The analogous statement for non-simply connected manifolds mentioned after Proposition~\ref{sis.bound} is proven analogously.

\begin{proof}[Proof of Proposition~\ref{sis.bound.fund}]
Let $(W,F)$ be the spin bordism from $(M, c_M)$ to $(N,f)$.
First, we use $0$-dimensional surgery in order to make $W$ connected. We will abuse the notation and also denote the bordism after surgery $(W,F)$. Note that $\Gamma$ is always finitely presented. Then, again by $0$-dimensional surgery, we change $W$ and $F$ such that $F$ induces a surjection on $\pi_1$.
Next, we use $1$-dimensional surgery such that the resulting $F$ induces an injection on $\pi_1$. 

As a consequence, the resulting map $F$ induces a bijection on $\pi_1$.  Next, we use $2$-dimensional surgeries to kill $\pi_2(W,M)$, cp. \cite[Proof of Prop.~2.1.1]{hebestreit_joachim}. This can always be achieved as every element of $\pi_2(W,M)$ then comes from an element in $\pi_2(W)$ and thus can be represented by an embedded $S^2$. The condition that $W$ is spin implies that this embedded $S^2$ has trivial normal bundle. Then, the embedding $M\hookrightarrow W$ is $2$-connected and $N\hookrightarrow W$ is $1$-connected. 
Thus, we can obtain $N$ from $M$ by attaching handles of dimensions $\ell$ with  $2\leq \ell \leq m-2$. Together with Theorem~\ref{surgthm} we then obtain the claim.
\end{proof}

\appendix
\section{Weak partial differential inequalities}

In this appendix, we recall the connection between viscosity solutions and distributional solutions of 
weak partial differential inequalities. All functions in this appendix are real-valued.

Let $\Delta=\d^*\d$ be the geometric Laplacian on functions on a Riemannian manifold $(M,g)$. Assume that $P$ is an operator of the form $P= a\Delta +V$ where $a$ is a positive smooth function
and $V$ is a continuous function. Let $f$ be a continuous function.

We say that 
  $$Pu\leq f$$
holds in the distributional sense if for all compactly supported smooth nonnegative functions $v$ on $M$
\[ \int_M uPv\, \vo_g\leq \int_M fv\, \vo_g.\]

We say that 
   $$Pu\leq f$$
 holds in the viscosity sense if $u$ is continuous and for every $p\in M$ and $\epsilon >0$ there is a neighborhood $U_\epsilon$ of $p$ and a $C^2$-function $h_\epsilon: U_\epsilon\to \mathbb{R}$ such that $h_\epsilon (p)=u(p), h_\epsilon \leq u$ in $U_\epsilon$ and $P h_\epsilon(p)\leq f(p)$.
 
\begin{thm}\label{viscosity-implies-distrib}\cite[Thm.~1 and 2]{Ishii} A  continuous function $u$ fulfills
 $Pu\leq f$  in the distributional sense if and only if it also fulfills the inequality in the viscosity sense.
\end{thm}
Actually, in the reference \cite[Thm.~1 and 2]{Ishii} the statement is proven for a wide class of second-order operators on $\mR^n$. But since this class includes the representation of $a\Delta_M$ in a chart of geodesic coordinates the above theorem follows.

\begin{corollary}\label{viscosity-implies-distrib_cor} Let $f$ be a nonnegative continuous function.
 Let $u\geq 0$ be a continuous function such that  $Pu\leq f$ in the classical sense whenever $u$ is positive. Then, $u$ fulfills $Pu\leq f$ in the distributional sense. 
\end{corollary}

\begin{proof}
 We see that $Pu\leq f$ in the viscosity sense by taking $h_\epsilon=u$ whenever $u$ is positive and $h_\epsilon=0$ otherwise. Then,  Theorem~\ref{viscosity-implies-distrib} implies the corollary.
\end{proof}

\bibliographystyle{acm}

\end{document}